\documentclass[11pt]{article}
\usepackage{geometry}                
\usepackage{graphicx}
\usepackage{amssymb,amsmath,amsfonts,amsthm}
\usepackage{mathtools}
\usepackage{enumerate}

\usepackage{natbib}
\usepackage{ifthen}
\usepackage[colorinlistoftodos, textwidth=2cm, obeyFinal]{todonotes}
\usepackage[colorlinks=true, pdfstartview=FitV, linkcolor=black, citecolor=black, urlcolor=black]{hyperref}
\usepackage{setspace}

\graphicspath{{./figs/}}

\newtheorem{theorem}{Theorem}
\newtheorem{lemma}{Lemma}

\newtheorem{proposition}{Proposition}
\theoremstyle{definition}
\newtheorem{definition}{Definition}

\newtheorem{example}{Example}
\newtheorem{remark}{Remark}

\newcommand{\app}{\mathrm{app}}
\newcommand{\A}{\mathcal{A}}
\newcommand{\D}{\mathcal{D}}
\newcommand{\R}{\mathbb{R}}
\newcommand{\Ind}{\mathbf{1}}
\newcommand{\N}{\mathbb{N}}
\newcommand{\I}{\mathcal{I}}
\newcommand{\thd}{\eta}
\newcommand{\E}[2][n]{\mathbb{E}\SwitchBracketsizeLeft{#1}\LeftBracketSize[#2\SwitchBracketsizeRight{#1}\RightBracketSize]}
\newcommand{\Prob}[2][n]{\mathbb{P}\SwitchBracketsizeLeft{#1}\LeftBracketSize\{#2\SwitchBracketsizeRight{#1}\RightBracketSize\}}

\newcommand{\argmin}{\mathrm{argmin}}

\DeclareMathOperator{\BV}{BV}
\DeclareMathOperator{\TV}{TV}

\allowdisplaybreaks

\newcommand{\abs}[2][n]{\SwitchBracketsizeLeft{#1}\LeftBracketSize\lvert#2\SwitchBracketsizeRight{#1}\RightBracketSize\rvert}
\newcommand{\norm}[2][n]{\SwitchBracketsizeLeft{#1}\LeftBracketSize\lVert#2\SwitchBracketsizeRight{#1}\RightBracketSize\rVert}
\newcommand{\set}[3][a]{\SwitchBracketsizeLeft{#1}\LeftBracketSize\{#2 : #3\SwitchBracketsizeRight{#1}\RightBracketSize\}}

\newcommand{\NextScriptStyle}[1]{{\scriptstyle{#1}}}
\newcommand{\NextScriptScriptStyle}[1]{{\scriptscriptstyle{#1}}}
\newcommand{\NextTextStyle}[1]{{\textstyle{#1}}}
\newcommand{\NextDisplayStyle}[1]{{\displaystyle{#1}}}
\newcommand{\SwitchBracketsizeLeft}[1]{
  \ifthenelse{\equal{#1}{b}\OR\equal{#1}{big}}{\let\LeftBracketSize=\bigl}{
    \ifthenelse{\equal{#1}{B}\OR\equal{#1}{Big}}{\let\LeftBracketSize=\Bigl}{
      \ifthenelse{\equal{#1}{g}\OR\equal{#1}{bigg}}{\let\LeftBracketSize=\biggl}{
    \ifthenelse{\equal{#1}{G}\OR\equal{#1}{Bigg}}{\let\LeftBracketSize=\Biggl}{
      \ifthenelse{\equal{#1}{s}\OR\equal{#1}{small}}{\let\LeftBracketSize=\NextScriptStyle}{
        \ifthenelse{\equal{#1}{ss}}{\let\LeftBracketSize=\NextScriptScriptStyle}{
          \ifthenelse{\equal{#1}{t}\OR\equal{#1}{text}}{\let\LeftBracketSize=\NextTextStyle}{
        \ifthenelse{\equal{#1}{d}\OR\equal{#1}{display}}{\let\LeftBracketSize=\NextDisplayStyle}{
          \ifthenelse{\equal{#1}{a}\OR\equal{#1}{auto}}{\let\LeftBracketSize=\left}{
            \let\LeftBracketSize=\relax}}}}}}}}}}
\newcommand{\SwitchBracketsizeRight}[1]{
  \ifthenelse{\equal{#1}{b}\OR\equal{#1}{big}}{\let\RightBracketSize=\bigr}{
    \ifthenelse{\equal{#1}{B}\OR\equal{#1}{Big}}{\let\RightBracketSize=\Bigr}{
      \ifthenelse{\equal{#1}{g}\OR\equal{#1}{bigg}}{\let\RightBracketSize=\biggr}{
    \ifthenelse{\equal{#1}{G}\OR\equal{#1}{Bigg}}{\let\RightBracketSize=\Biggr}{
      \ifthenelse{\equal{#1}{s}\OR\equal{#1}{small}}{\let\RightBracketSize=\NextScriptStyle}{
        \ifthenelse{\equal{#1}{ss}}{\let\RightBracketSize=\NextScriptScriptStyle}{
          \ifthenelse{\equal{#1}{t}\OR\equal{#1}{text}}{\let\RightBracketSize=\NextTextStyle}{
        \ifthenelse{\equal{#1}{d}\OR\equal{#1}{display}}{\let\RightBracketSize=\NextDisplayStyle}{
          \ifthenelse{\equal{#1}{a}\OR\equal{#1}{auto}}{\let\RightBracketSize=\right}{
            \let\RightBracketSize=\relax}}}}}}}}}}

\title{Multiscale Change-point Segmentation: Beyond Step Functions}
\author{Housen Li, Qinghai Guo, and Axel Munk}

\begin{document}

\setstretch{1.2}

\maketitle

\begin{abstract}
Modern multiscale type segmentation methods are known to detect multiple change-points with high statistical accuracy, while allowing for fast computation. Underpinning theory has been developed mainly for models that assume the signal as a piecewise constant function. In this paper this will be extended to certain function classes beyond such step functions in a nonparametric regression setting, revealing certain multiscale segmentation methods as robust to deviation from such piecewise constant functions. Our main finding is the adaptation  over such function classes for a universal thresholding, which includes  bounded variation functions, and (piecewise) H\"{o}lder functions of smoothness order $ 0 < \alpha \le1$ as special cases. From this we derive statistical guarantees on feature detection in terms of jumps and modes. Another key finding is that these multiscale segmentation methods perform nearly (up to a log-factor) as well as the oracle piecewise constant segmentation estimator (with known jump locations), and the best piecewise constant approximants of the (unknown) true signal. Theoretical findings are examined by various numerical simulations. 
\end{abstract}

\noindent {\it MSC 2010 subject classifications:}
62G08, 62G20, 62G35.

\noindent {\it Key words and phrases:}
{Change-point regression,} adaptive estimation, oracle inequality, jump detection,  model misspecification, multiscale inference, {approximation spaces}, {robustness.}

\section{Introduction}\label{s:intro}

Throughout we assume that observations are given through the regression model
\begin{equation}\label{eq:model}
y_{i}^{n} = \bar f_i^n + \xi_{i}^n,\qquad i = 0, \ldots, n-1,
\end{equation}
where $\bar f_i^n= n\int_{[i/n,(i+1)/n)} f(x) dx$, and $\xi^n = (\xi_{0}^{n},\ldots,\xi_{n-1}^{n})$ are independent {(not necessarily i.i.d.)} centered sub-Gaussian random variables with scale parameter $\sigma$, that is,
\begin{equation*}
\E[a]{e^{u\xi_{i}^{n}}} \le e^{u^2\sigma^2/2},\qquad \text{ for every } u \in \R.
\end{equation*}
For simplicity, the scale parameter (i.e.~noise level in the Gaussian case) $\sigma$ in model~\eqref{eq:model} is assumed to be known, as it can be easily pre-estimated {$\sqrt{n}$-consistently from the data}, {which does not affect our results}, see e.g.~\citet{DM98,ACLZ14,DeWi16,TGM17}. 

{In this paper we are concerned with potentially discontinuous signals {$f: [0,1)\to \R$} in~\eqref{eq:model}. {As a minimal condition,} we always} assume that the underlying (unknown) signal $f$ lies in $\D \equiv \D([0,1))$, the space of c\`{a}dl\`{a}g functions on $[0,1)$, which are right-continuous and have left-sided limits~\citep[cf.][Chapter 3]{Bil99}. In~\eqref{eq:model}, we embed for simplicity the sampling points $x_{i,n}=i/n$ equidistantly in the unit interval. {However, we stress that all our results can be transferred to more general {domains ($\subseteq\R$) and} sampling schemes, also for random $x_{i,n}$.} 

{For} the particular case that $f$ is piecewise constant {with a finite but unknown number of jumps}, model~\eqref{eq:model} {has been of particular interest throughout the past and} is often referred to as \emph{change-point regression model}. {The} related problem of estimating the number, locations and sizes of change-points {(i.e.~its {locations of} discontinuity)} has a long and rich history in {the statistical} literature, {see e.g.~\cite{IbrHas81} or~\cite{KorKor11} for some selective textbook references on statistical efficient estimation of a \emph{single} change-point.} {\cite{Tuk61} already phrased the problem of segmenting a data sequence into constant pieces as {the} ``regressogram problem'' and it occurs in a plenitude of applications.} {From a risk minimization point of view it is well known that certain Bayesian estimators are (asymptotically) optimal \citep{IbrHas81}; however, they are not feasible from a computational point of view, particularly when it comes to multiple change-point recovery.} Therefore, recent years have witnessed a renaissance in change-point inference motivated by several applications which require {\it computationally fast} and {\it statistically efficient} finding of potentially \emph{many} change-points {in large data sets}, see e.g.~\cite{OlsVenLucWig04}, \cite{Sie13} and~\cite{BHM16} for its relevance to {cancer} genetics, \cite{HaoZha15} for network analysis, {\cite{ACLZ14} for econometrics,} and~\cite{Hot13} for electrophysiology, {to name a few.} This challenges {statistical} methodology due to the multiscale nature {of these problems} (i.e.~change-points occur at different {e.g.~temporal} scales) {due to a potentially large number of change-points,} and due to the large number of data points (a few millions or more), {requiring computationally efficient methods}. {Furthermore, it is of great practical relevance to have} change-point segmentation methods that provide statistical certificates of evidence for {its} findings, such as uniform confidence sets for the change-point locations or jump sizes~\citep{FriMunSie14,LMS16}. 

Computationally efficient segmentation methods which provide at the same hand certain statistical guarantees have been lately proposed, which are either based on dynamic programming~\citep{BoyKemLieMunWit09,KillFeaEck12,FriMunSie14, DKK15,LMS16,MaToGuPa16,HayEckFer16}, local search~\citep{ScoKno74,OlsVenLucWig04,Fry14} or convex optimization~\citep{HarLev08,TibWan08,HarLev10}. 

{Typically, the statistical justification {for all} the aforementioned methods is given for} models which assume that the underlying truth is a piecewise constant function with a fixed but unknown number of changes. For extensions to increasing number of changes of the truth (as the number of observations increases), see e.g.~\citep{ZhSie12,Fry14}, {or under an additional sparsity assumption~\citep{CJL12}.} However, in general, nothing is known for {such} segmentation methods in the general nonparametric regression setting as in~\eqref{eq:model} when~$f$ is {not a piecewise constant function,} e.g.~a smooth function. Notable exceptions are the jump-penalized least square estimator in~\citep{BoyKemLieMunWit09}, and the unbalanced Haar wavelets based estimator in~\citep{Fry07}, for which the $L^2$-risk has been analyzed for functions which can be sufficiently fast approximated by piecewise constant functions (in our notation this corresponds to functions in the space  $\A_2^\gamma$, see section~\ref{ss:approx} for the definition).

Intending to fill such a gap, we provide a comprehensive risk analysis for a range of multiscale change-point methods when $f$ in~\eqref{eq:model} is not a change-point function. To this end, we introduce {in a first step} a {general} class of \emph{multiscale change-point segmentation methods}, {with scales specified by general $c$-normal systems (adopted from~\citet{Nem85}, see Definition~\ref{df:cn}),} unifying several previous methods. {This includes} particularly the simultaneous multiscale change-point estimator (SMUCE), which is introduced by~\cite{FriMunSie14} via minimizing the number of change-points under a side constraint that is based on a simultaneous multiple testing procedure on all scales (length of subsequent observations), and related estimators which are built on different multiscale systems~\citep{Wal10}, {or penalties~\citep{LMS16}.} {These methods can be viewed also as a natural multiscale extension of certain jump penalized estimators via convex duality~\citep[see][]{BoyKemLieMunWit09,KillFeaEck12}.} {Implemented} by {accelerated} dynamic programming algorithms, {these methods often have a runtime $O(n \log n)$, and are} found empirically promising in various applications~\citep[see e.g.][]{Hot13,FutHotMunSie13,BehMun15,KillFeaEck12}. In case that $f$ in model~\eqref{eq:model} is a step function, the {statistical} theory {for these methods} is well-understood meanwhile. {For example, minimax optimality of estimating the change-point locations and sizes has been shown, which is based on} exponential deviation bounds on the number, and the locations of change-points. {Furthermore, these methods also obey} optimal {minimax} detection {properties} (in the sense of testing) of vanishing signals {as well}, {and provide honest simultaneous confidence statements for all unknown quantities}~\citep[see][]{FriMunSie14,LMS16,PSM16}. 

To complement the understanding of these methods, this work aims to analyze their behavior when the true regression function $f$ is beyond a piecewise constant function. To this end, we derive {a)} convergence rates for sequences of piecewise constant functions with increasing number of changes {(Theorem~\ref{th:step})}, and {b)} for functions in {certain} approximation spaces {(Theorem~\ref{th:approx})}, well-known in approximation theory, cf.~\citet{DeLo93}, {(see Section~\ref{s:theory})}, generalizing the above mentioned results results for quadratic risk to general $L^p$-risk ($0 < p < \infty$).   {As a consequence,} we obtain the {minimax} optimal rates $n^{-2/3\cdot\min\{1/2,1/p\}}$ and $n^{-2\alpha/(2\alpha+1)\cdot\min\{1/2,1/p\}}$ (up to a log-factor) in terms of $L^p$-loss both almost surely and in expectation for the cases that $f$ has bounded variation~{\citep[see][]{MamGee97}}, and that $f$ is (piecewise) H\"{o}lder continuous of order $0 < \alpha \le 1$, respectively. Most importantly, the discussed multiscale change-point segmentation methods are universal (i.e.~independent of the smoothness assumption of the unknown truth signal), as the only tuning parameter {(which can be thought of as a universal threshold)} can be chosen as $\thd \asymp \sqrt{\log n}$. {We will show that} for this choice, these methods automatically \emph{adapt} to the unknown ``smoothness'' of the underlying function in an optimal way, no matter whether it is piecewise constant or it lies in {the aforementioned} function spaces. As an illustration, we present the performance of SMUCE~\citep{FriMunSie14} with universal parameter choice $\thd = 0.42 \sqrt{\log n}$, on {different signals}  in Figure~\ref{over_fig}. It clearly shows that SMUCE, although {designed to provide a piecewise constant solution,} successfully recovers the shape of {all} underlying signals no matter whether they are locally constant or not, as suggested by our theoretical findings.

\begin{figure}[!ht]
\centering
\includegraphics[width=0.9\textwidth]{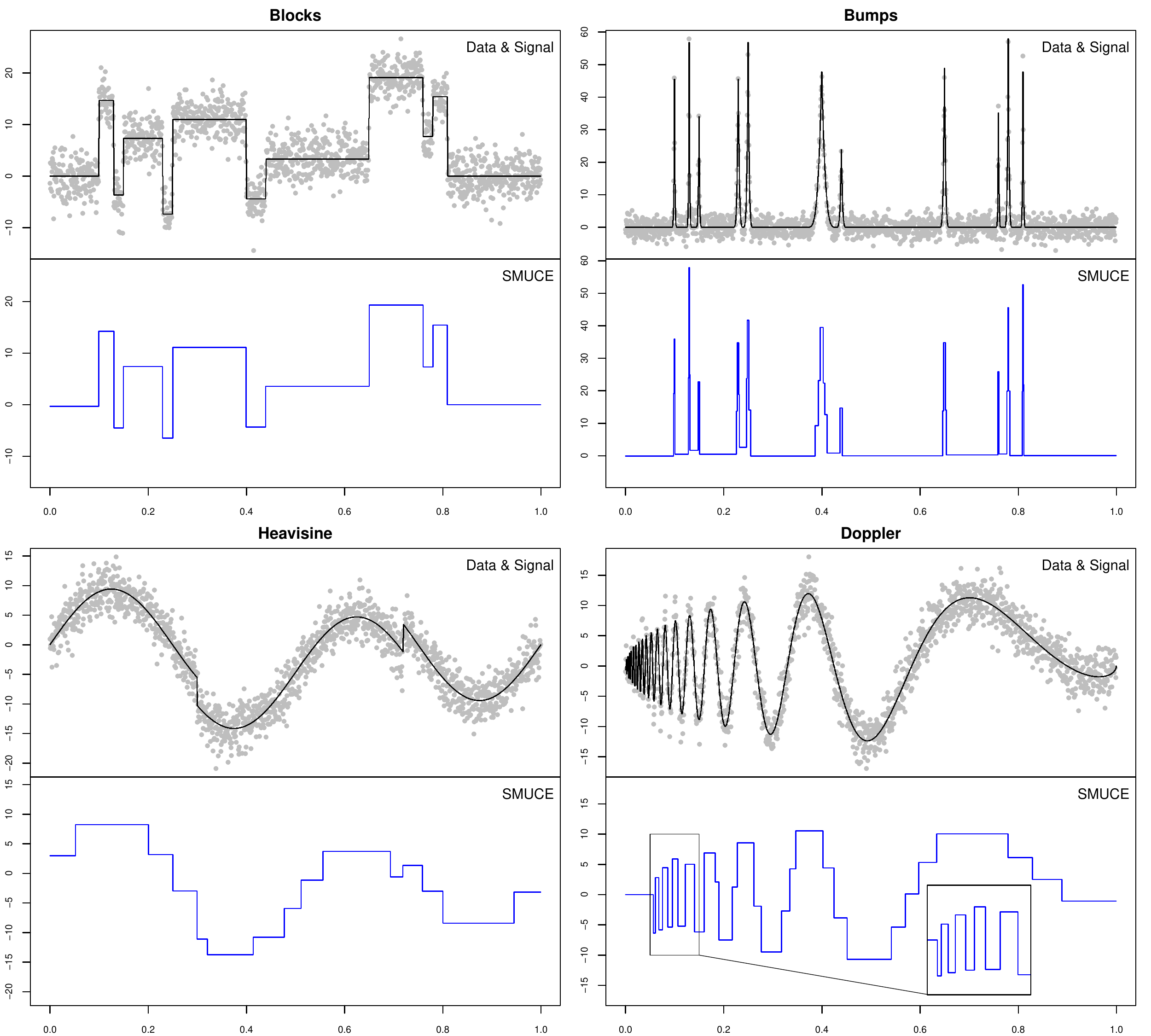}
\caption{Estimation by {the multiscale change-point segmentation method SMUCE} \citep{FriMunSie14} for Blocks, Bumps, Heavisine, and Doppler signals~\citep{DonJoh94} with sample size $n = \text{1,500}$, and signal-to-noise ratio $\|f\|_{L^2}/\sigma = 3.5$.}
\label{over_fig}
\end{figure}

Indeed, the derived convergence rates allow us to derive statistical guarantees for such feature detection, see Section~\ref{s:feature}. More precisely, we show for general (incl.~piecewise constant) signals in approximation spaces that the discussed methods recover no less jumps and modes (or troughs) than the truth as the sample sizes tends to infinity; This statement should be interpreted with the built-in parsimony (i.e.,~minimization of number of jumps) of these methods, which suggests that the number of artificial jumps and modes (or troughs) is ``minimal''; {At the same hand}, large increases (or decreases) of the discussed estimators imply increases (or decreases) of the true signal with high confidence; (Theorem~\ref{th:feature}). In Figure~\ref{fig:feature}, based on our theoretical finding, one can claim, for example, that the two \emph{large} jumps (marked by solid vertical lines) are significant with confidence at least $90\%$ (see Remark~\ref{rem:ft}). In the particular case of step signals, we further show the consistency in estimating the number of jumps, and an error bound of the best known order (in terms of sample sizes) on the estimation accuracy of change-point locations {(Proposition~\ref{pp:feature})}.

\begin{figure}
\centering
\includegraphics[width=0.8\textwidth,clip]{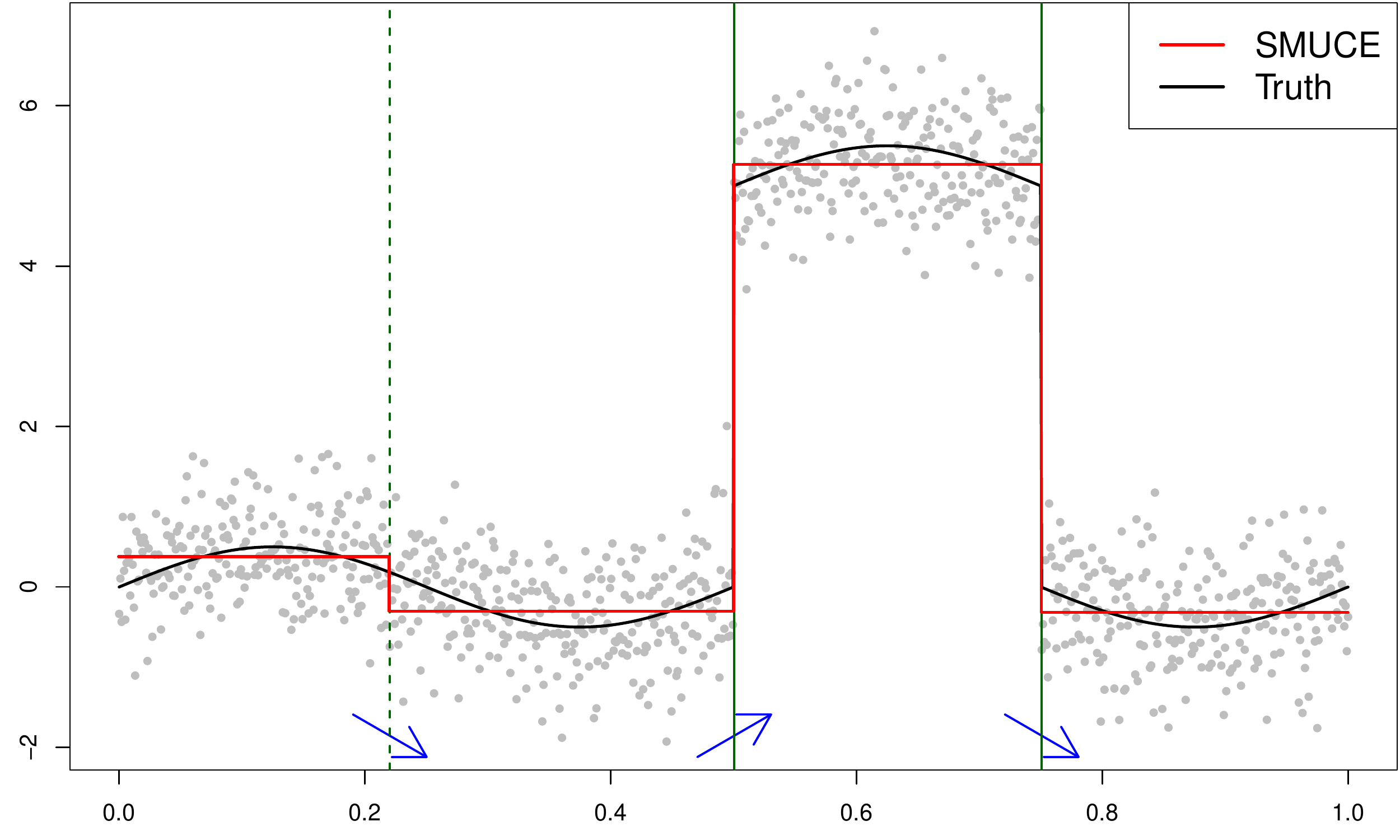}
\caption{Feature detection by SMUCE with universal threshold $\eta(0.1)$ by~\eqref{eq:refineQ} (sample size $n=1,000$, $\text{SNR} = 5$). The solid vertical lines mark significant jumps, while the dashed one marks an insignificant jump; and the arrows at the bottom indicate significant increases and decreases; with simultaneous confidence at least $90\%$. See {Remark~\ref{rem:ft} in} Section~\ref{s:feature} for details.
\label{fig:feature}} 
\end{figure}

Finally, we address the issue how to benchmark properly the investigated methods. To this end, we show that the multiscale change-point segmentation methods perform nearly no worse than piecewise constant segmentation estimators whose change-point locations are provided by an oracle. By considering such oracles, we discover a \emph{saturation} phenomenon {(Theorem~\ref{th:oSeg} and Example~\ref{ex:sat})} for the class of all piecewise constant segmentation estimators: only the suboptimal rate $n^{-2/3}$ is attainable for smoother functions in H\"older classes with $\alpha> 1$. From a slightly different perspective, we show that the multiscale change-point segmentation methods  perform nearly as well as the best {(deterministic)} piecewise constant approximant of the {true} signal with the same number of jumps or less {(Proposition~\ref{pp:bestApp}).} 

Besides such theoretical interest~\citep[cf.~also][]{LinSeo14,Far14}, the study on models beyond piecewise constant functions is also of particular practical importance, since a piecewise constant function is {actually known to be} only an approximation of the underlying signal in many applications. {For example, in} DNA copy number analysis, for which the change-point regression model {with locally constant signal} is commonly assumed~\citep[see e.g.][]{OlsVenLucWig04,Lai05},  a periodic trend distortion {with small amplitude} (known as genomic waves) {is well known to be present}~\citep{DLHYGHBMW08}. Thus {our work} can be {also} regarded as examination of the robustness of {such} segmentation methods against model misspecification. {We} consider a piecewise constant estimator as robust, if it recovers the majority of interesting features of the underlying true regression function with as small number of jumps as possible. For instance, Figure~\ref{robust_fig} shows the performance of SMUCE on a typical signal from DNA copy number analysis, where a locally constant function is slightly perturbed, in cases of different noise levels. Visually, SMUCE {seems to recover the} major features, and the recovered signal provides a simple yet informative summary of the data, meanwhile staying close to the true signal, which confirms our theoretical findings. We note that our viewpoint here {complements} a recent work by~\cite{SBK16} who considered a \emph{reverse} scenario: a sequence of smooth functions approaches a step function in the limit.  

\begin{figure}[!ht]
\centering
\includegraphics[width=0.9\textwidth]{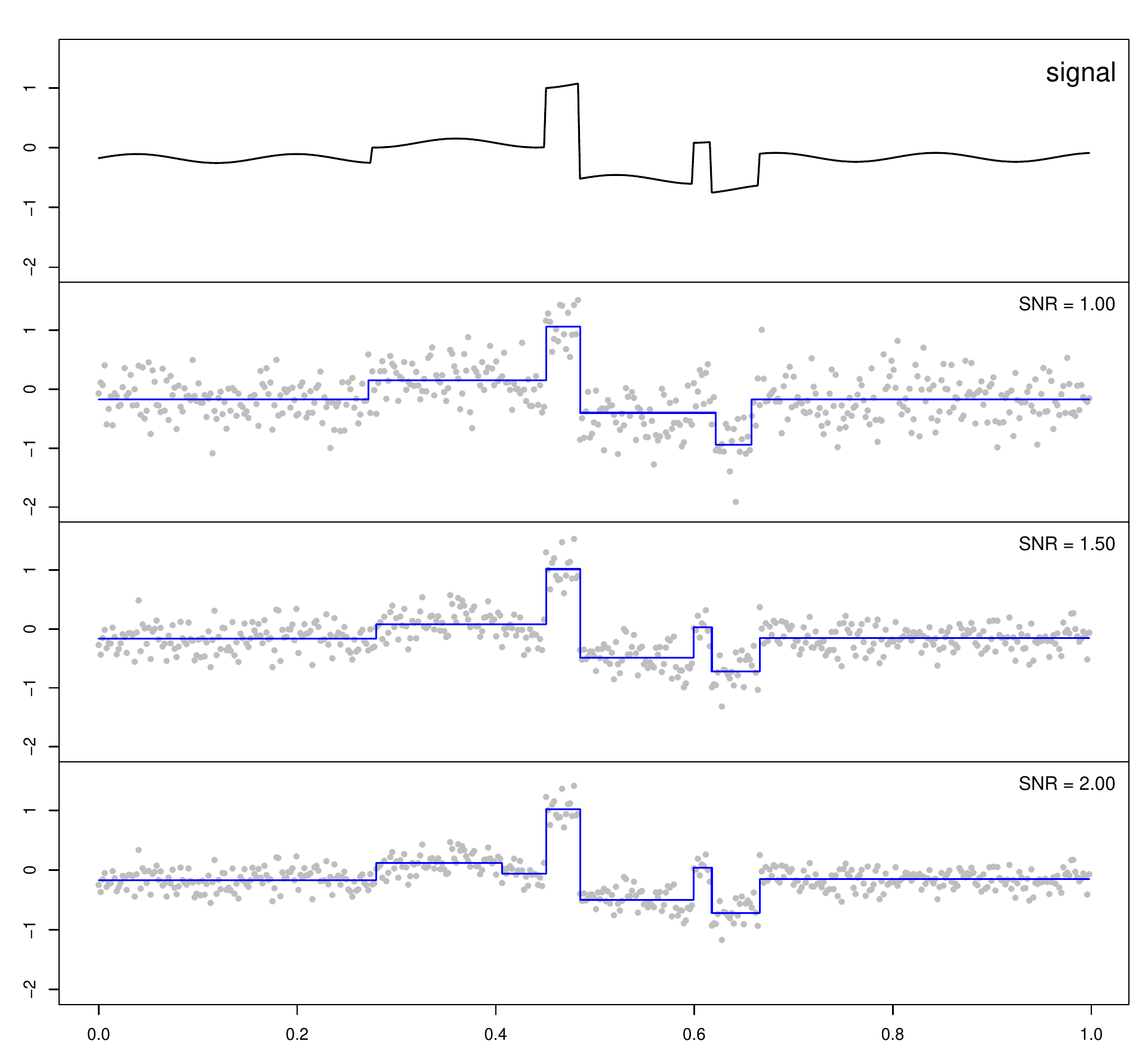}
\caption{{Estimation by {SMUCE} for the signal in~\cite{OlsVenLucWig04} and~\cite{ZhaSie07} with various signal-to-noise ratios $\|f\|_{L^2}/\sigma$, see also Section~\ref{s:numerics}.}
\label{robust_fig}} 
\end{figure}

In summary, we show that a large class of multiscale change-point segmentation methods with a universal parameter choice are \emph{adaptively minimax optimal} (up a log-factor) for step signals (possibly with unbounded number of change-points) and for {(piecewise) smooth}  signals in certain approximation spaces {(Theorems \ref{th:step} and \ref{th:approx})} for general $L^p$-risk. Building on this, we obtain statistical guarantees on feature detection, such as recovery of the number of discontinuities, or modes {(Proposition~\ref{pp:feature} and Theorem~\ref{th:feature})}, which explain well-known empirical findings.  In addition, we show oracle inequalities for such multiscale change-point segmentation methods in terms of both segmentation and approximation of the true signal {(Theorem~\ref{th:oSeg} and Proposition~\ref{pp:bestApp})}.

The paper is organized as follows. In Section~\ref{s:method}, we introduce {a general class of} multiscale change-point segmentation methods, {discuss examples} and {provide technical} assumptions. We derive uniform bounds on the $L^p$-loss over step functions with possibly increasing number of change-points and over certain approximation spaces in Section~\ref{s:theory}, and present their implication on feature detection in Section~\ref{s:feature}. Section~\ref{s:oracle} focuses on the oracle properties of multiscale change-point segmentation methods from a segmentation and an approximation perspective, respectively. Our theoretical findings are {investigated for finite samples by a simulation study} in Section~\ref{s:numerics}. The paper ends with a conclusion in Section~\ref{s:discuss}. Technical proofs are collected in the appendix.

\section{Multiscale change-point segmentation}\label{s:method}
{Recall model~\eqref{eq:model} and let $f$ now in} $\mathcal{S} \equiv \mathcal{S}([0,1))$, the space of right-continuous step functions $f$ on $[0,1)$ {with a finite (but possibly unbounded) number of jumps,} that is, {for some} $k \in \N$
\begin{equation}\label{eq:step}
f = \sum_{i = 0}^k c_i\Ind_{[\tau_i,\tau_{i+1})}\quad \text{with }0 = \tau_0 < \ldots < \tau_{k+1} = 1, \text{ and }c_i \neq c_{i+1} \text{ for each }i. 
\end{equation}
Here $J(f) \coloneqq \{\tau_1,\ldots,\tau_k\}$ denotes the set of change-points of $f$. By \emph{intervals} we always refer to those of the form $[a,b)$, $0 \le a < b \le 1$. {In the following we introduce a general class of multiscale change-point estimators comprising various methods recently developed. To this end, we fix a system $\mathcal{I}$ of subintervals of $[0,1)$ in the first step. Given $\mathcal{I}$, }
we introduce a {general} class of \emph{multiscale change-point segmentation} {estimators} $\hat{f}_n$~\citep[see][]{FriMunSie14,LMS16,PSM16} as a solution to {the (nonconvex) optimization problem}
\begin{equation}\label{eq:mr_segment}
\min_{f \in \mathcal{S}} \#J(f) \qquad \text{ subject to } T_{\mathcal{I}}(y^n; f) \le \thd.
\end{equation}
Here $y^n \coloneqq \{y^n_i\}_{i=0}^{n-1}$ {denotes the observational vector, and} $\thd\in\R$ is a {universal} threshold {to be defined later}. {The side constraint in~\eqref{eq:mr_segment} is defined by} a multiscale test statistic 
\begin{equation*}
T_{\mathcal{I}}(y^n; f) \coloneqq \sup_{\substack{I \in \mathcal{I} \\ f \equiv c_I \text{ on } I}}\left\{ \frac{1}{\sqrt{n\abs{I}}} \abs[B]{\sum_{i/n \in I} (y_i^n - c_I)} - s_I\right\},
\end{equation*}
with $s_I \in \R$ {a} scale penalty, which can be deterministic or random, and might even depend on the candidate $f$ and the data $y^n$. Note, that the solution to the optimization problem~\eqref{eq:mr_segment} {always exists but} might be \emph{non-unique}, in which case one could pick an arbitrary solution.

The side constraint in~\eqref{eq:mr_segment} originates from testing {simultaneously the residuals of a candidate $\hat f$ with values $c_I$ on the multiscale system $\mathcal{I}$. In model~\eqref{eq:model} under a Gaussian error, this} combines all the local {likelihood ratio} tests whether the local mean $\bar{f}_I$ of $f$ on $I$ equals to a given $c_I$ for every $I \in \mathcal{I}$. {Hence,} this provides a criterion for {testing} the constancy of $f$ on each of its segments in $\mathcal{I}$ {\citep[for a detailed account see][]{FriMunSie14}}. The choice of the scale penalties $s_I$ determines the estimator. It balances the detection power over different scales, see~\cite{DueSpok01}, \citet{Wal10} and \citet{FriMunSie14} for several choices, and~\cite{DavHoeKra12} for the unpenalized estimators, $s_I\equiv 0$, {in a slightly different model.} Thus, {any} multiscale change-point segmentation method amounts to search for the most parsimonious candidate over the acceptance region of the multiple tests {on the right hand side in~\eqref{eq:mr_segment}} {performed over the system $\mathcal{I}$}. The threshold $\thd$ in~\eqref{eq:mr_segment} {provides} a trade-off between data-fit and parsimony, {and} can be chosen such that the truth $f$ {satisfies} the side constraint  with a pre-specified probability $1-\beta$. To this end, $\thd\equiv \thd(\beta)$ is {chosen as} the upper ($1-\beta$) quantile of the distribution of $T_{\mathcal{I}}(\xi^n; 0)$, which can be determined by Monte-Carlo simulations or  asymptotic considerations~\citep{FriMunSie14,PSM16}. Then the choice of significance level $\beta$ provides an upper bound on the family-wise error rate of the aforementioned multiple test. It immediately {provides for $\hat f_n$} a control of overestimating the number of jumps $\#J(f)$ of $f$, i.e.
\[
\Prob[a]{\#J(\hat{f}_n) \le \#J(f)} \ge 1- \beta\qquad\text{ uniformly over all } f\in \mathcal{S}. 
\]
Also, {with a different penalty,} it is possible to control instead the false discovery rate by means of \emph{local} quantiles, see~\citet{LMS16} for details. {Comprising the above mentioned choices for $\mathcal{I}$, we will see that}, if the system of intervals $\mathcal{I}$ is rich enough, {for the asymptotic analysis of {all} these estimators} it is sufficient to work with a universal threshold $\thd \asymp \sqrt{\log n}$ in~\eqref{eq:mr_segment} (see Section~\ref{s:theory}). 

The system $\mathcal{I}$ will be required to be truly \emph{multiscale}, i.e.~the multiscale change-point segmentation methods in~\eqref{eq:mr_segment} require the associated interval system  $\mathcal{I}$ to contain different scales, the richness of which can be characterized by the concept of \emph{normality}.

\begin{definition}[\cite{Nem85}]\label{df:cn}
A system $\mathcal{I} \equiv \mathcal{I}_n$ of intervals is called \emph{normal} (or \emph{$c$-normal}) for some constant $c > 1$, provided that it satisfies the following requirements.
\begin{itemize}
\item[(i)]
For every interval $I \subseteq [0,1)$ with length $\abs{I} > c/n$, there is an interval $\tilde{I}$ in $\mathcal{I}$ such that $\tilde{I} \subseteq I$ and $\abs{\tilde{I}} \ge c^{-1}\abs{I}$. 
\item[(ii)]
The end-points of each interval in $\mathcal{I}$ lie on the grid $\set{i/n}{i = 0,\ldots,n-1}$.
\item[(iii)]
The system $\mathcal{I}$ contains {all} intervals $[i/n, (i+1)/n)$, $i = 0, \ldots, n-1$. 
\end{itemize}
\end{definition}

\begin{remark}[Normal systems]
The requirement (i) in the above definition is crucial, while (ii) and (iii) are {of technical nature due to the discrete} sampling locations $\{i/n\}_{i=0}^{n-1}$ {and can be generalized}. Examples of normal systems include the highly redundant system $\mathcal{I}^0$ of all  intervals whose end-points lie on the grid {$\{i/n\}_{i=0}^{n-1}$}~\citep[suggested by e.g.][]{SieYak00,DueSpok01,FriMunSie14} {of order $O(n^2)$}, and less redundant but still asymptotically efficient systems~\citep{DavKov01,Wal10,RivWal12}, {typically of order $O(n\log n)$.} {Remarkably}, there are even normal systems with cardinality of order ${O(n)}$, such as the \emph{dyadic partition system} 
\[
\set[g]{\Bigl[\frac{i}{n}\lceil 2^{-j}n \rceil, \frac{i+1}{n}\lceil 2^{-j}n \rceil\Bigr)}{i = 0, \ldots, 2^j-1,\, j=0,\ldots,\lfloor\log_2n\rfloor},
\]
which {can be shown to be} $2$-normal, see \citet{GrLiMu15}. 
\end{remark}

{
\begin{definition}[Multiscale change-point segmentation estimator]\label{def:mrseg}
Any estimator satisfying~\eqref{eq:mr_segment} is denoted as a \emph{multiscale change-point segmentation estimator}, if 
\begin{itemize}
\item[(i)]
the interval system $\mathcal{I}$ is $c$-normal for some constant $c > 1$; 
\item[(ii)]
the scale penalties $s_I$ satisfy almost surely that
\[
\sup_{I \in \mathcal{I}} \abs{s_I} \le \delta \sqrt{\log n}\qquad \text{ for some constant }\delta > 0.
\]
\end{itemize}
\end{definition}
}

\begin{remark}[Some multiscale segmentation methods]For sub-Gaussian error~$\xi^n$
\[
\sup_{I \in \mathcal{I}} \frac{1}{\sqrt{n\abs{I}}} \abs[B]{\sum_{i/n \in I} \xi_i^n}
\] 
is at most of order $\sqrt{\log n}$~\citep[see e.g.][]{Sha95}, so Definition~\ref{def:mrseg} (ii) is quite natural. In particular, Definition~\ref{def:mrseg} (ii) includes many common scale penalties.  For instance, SMUCE~\citep{FriMunSie14} and FDRSeg~\citep{LMS16} are special cases. More precisely, for SMUCE, it amounts to select $\mathcal{I} = \mathcal{I}^0$, the system of all possible intervals,  and $s_I = \sqrt{2\log (e/\abs{I})}$, and for FDRSeg, the same system $\mathcal{I} = \mathcal{I}^0$ but a different scale penalty $s_I = \sqrt{2\log (e\abs{\tilde{I}}/\abs{I})}$ with $\tilde{I}$ being the constant segment, which contains $I$, of the candidate solution. {The case $s_I \equiv 0$ {is also included and} has been suggested by~\cite{DavHoeKra12}.}
\end{remark}

\section{Asymptotic error analysis}\label{s:theory}

This section mainly provides convergence rates of the multiscale change-point segmentation methods for the model~\eqref{eq:model} with equidistant sampling points. We stress, that the subsequent results can be easily generalized to non-equidistant (and random) sampling points $x_{i,n}$ under appropriate conditions on the design~\citep[see][]{MunDet98}; this is, however, suppressed to ease {presentation}.

\subsection{Convergence rates for step functions}\label{ss:step}
We consider first {locally constant} change-point {regression}, i.e.~the underlying signal $f \in \mathcal{S}$ in model~\eqref{eq:model}. We introduce the class of uniformly bounded piecewise constant functions (recall~\eqref{eq:step}) with up to $k$ jumps
\[
\mathcal{S}_{L}(k) \coloneqq \set[B]{f \in \mathcal{S}}{\#J(f) \le k, \text{ and } \norm{f}_{L^{\infty}} \le L},
\]
for $k \in \N_0$ and $L > 0$. If the number of change-points is bounded, {i.e.~$k$ is known beforehand}, the estimation problem is, roughly speaking, parametric, by interpreting change-point locations and function values as parameters. A rather complete analysis of this situation is provided either from a Bayesian viewpoint~\citep[see e.g.][]{IbrHas81,HuAn03} or from a likelihood viewpoint~\cite[see e.g.][]{YaoAu89,BraMueMue00,SieYak00,BoyKemLieMunWit09,KorKor11}. However, in order to understand {the increasing difficulty {of change-point estimation} as the number of change-points gets {larger},} i.e.~the nonparametric nature of change-point regression, we allow {now} the number of change-points to increase as the number of observations {tends} to infinity.  

\begin{theorem}[{Adaptation I}]\label{th:step}
Assume model \eqref{eq:model}. Let  $0 <  p,\, r < \infty$, and $k_n\in\N_0$ be such that $k_n = o(n)$ as $n \to \infty$. Then:
\begin{itemize}
\item[\emph{(i)}]
If $\hat f_n$ is a multiscale change-point segmentation estimator in Definition~\ref{def:mrseg} with constants $c$ and $\delta$, and threshold 
\begin{equation}\label{eq:defQ}
\thd \coloneqq a\sqrt{\log n}\qquad \text{ for some } a > \delta + \sigma\sqrt{2r+4}, 
\end{equation}
then the following upper bound holds
\begin{equation*}
\limsup_{n \to \infty} \frac{1}{\sqrt{\log n}} \Bigl(\frac{n}{2k_n+1}\Bigr)^{\min\{1/2, 1/p\}} \sup_{f \in \mathcal{S}_L(k_n)}\E[a]{\norm{\hat{f}_n - f}_{L^{p}}^r}^{1/r} < \infty.
\end{equation*}
The same result also holds almost surely if we drop the expectation $\E{\cdot}$. 
\item[\emph{(ii)}]
If noise $\xi_i^n$ in model \eqref{eq:model} has a density $\varphi_{i,n}$ such that for some constants $\sigma_0$ and
$z_0$
\begin{equation}\label{eq:nlb}
\max_{i,n}\int \varphi_{i,n}(x)\log\frac{\varphi_{i,n}(x)}{\varphi_{i,n}(x+z)} dx \le \frac{z^2}{\sigma_0^2} \qquad \text{ for } \abs{z} \le z_0
\end{equation}
then the following lower bound holds
\[
\liminf_{n \to \infty} \Bigl(\frac{n}{2k_n+1}\Bigr)^{\min\{1/2, 1/p\}}\inf_{\hat g_n } \sup_{f \in \mathcal{S}_L(k_n)} \E[a]{\norm{\hat{g}_n - f}_{L^{p}}^r}^{1/r} > 0,
\]
where the infimum is taken over all estimators $g_n$. 
\end{itemize} 
\end{theorem}

\begin{proof}
See Appendix~\ref{app:step}. 
\end{proof}

\begin{remark}\label{rem:adapt}
{
Note that condition~\eqref{eq:nlb} is a typical assumption for establishing lower bounds \citep[see e.g.][]{Tsy09}. In particular, if $\exp(-c_1 x^2) \lesssim \varphi_{i,n}(x) \lesssim\exp(-c_2 x^2)$ with constants $c_1,c_2$, then condition~\eqref{eq:nlb} holds for any $z_0 > 0$, e.g.~a Gaussian density.}  Theorem~\ref{th:step} states that multiscale change-point segmentation estimators are up to a log-factor adaptively minimax optimal over sequences of classes $\mathcal{S}_{L}(k_n)$ for all possible $k_n$ and $L$. A common choice of $k_n$ is  $k_n \asymp n^{\theta}$, $0 \le \theta < 1$, which in particular reproduces the convergence results in~\cite{LMS16}.  It also includes the case $\theta = 0$, where, by convention, $k_n \equiv k$ is bounded. Note that the universal threshold $\thd$ in~\eqref{eq:defQ} is independent of the specific loss function, provided that $a$ is large enough. Furthermore, one can relax such choice of $\thd$ by allowing $a = \delta + \sigma\sqrt{2r+4}$ in~\eqref{eq:defQ}, and even select 
\begin{equation}\label{eq:refineQ}
\thd= \thd(\beta)  \qquad \text{ with } \beta =  \mathcal{O}({n}^{-r}), 
 \end{equation}
see~Appendix~\ref{app:step}. By Shao's theorem~\citep{Sha95}, it is clear that $\thd(\beta) \le (\delta + \sigma\sqrt{2})\sqrt{\log n}$. {A more refined analysis is even possible, although not necessary for our purposes.}  For instance, in case of no scale penalization and $\mathcal{I}$ consisting of all intervals, it follows from~\cite{SieVen95} and \citet{Ka07} that 
\[
\thd(\beta)\sim \sqrt{2\log n}+ \frac{\log\log n + \log\frac{\lambda}{4\pi} -2\log\log(1/\beta)}{2\sqrt{2\log n}}\qquad\text{ as }n \to \infty,
\]  
with constant $\lambda \in (0,\infty)$. Note, finally, that the restriction $p< \infty$ {in Theorem~\ref{th:step}} is necessary and natural, because $L^{\infty}$-loss is not {reasonable in} change-point estimation problems \citep[as no estimator can detect change-point locations at a rate faster than $\mathcal{O}(1/n)$, see][which leads to inconsistency {of any estimator} with respect to $L^\infty$-loss]{ChaWal13}.
\end{remark}

\subsection{Robustness to model misspecification}\label{ss:approx}
{As discussed in the Section~\ref{s:intro}, {in} practical applications, it often occurs that the underlying signal $f$ in model~\eqref{eq:model} is only approximately piecewise constant. To address this issue, we next consider the $L^p$-loss of the multiscale change-point segmentation methods for more general functions. In order to characterize the degree of model misspecification, we {adopt} from {nonlinear} approximation theory~\citep[cf.][]{DeLo93,DeV98} the \emph{approximation spaces}~as
\[
\A^{\gamma}_q \coloneqq \set[B]{f \in  \D}{\sup_{k \ge 1}k^{\gamma}\Delta_{q,k}(f) < \infty},\quad \text{ for } 0 < q \le \infty, \,  \gamma > 0,
\]
where the approximation error $\Delta_{q,k}$ is defined as 
\begin{equation}\label{eq:approxErr}
\Delta_{q,k}(f) \coloneqq \inf \set[g]{\norm{f - g}_{L^q}}{g \in \mathcal{S}, \, \#J(g) \le k}.
\end{equation}
{Introduce the subclasses} 
\[
\A_{q,L}^{\gamma} \coloneqq \set[B]{f \in  \D}{\sup_{k \ge 1}k^{\gamma}\Delta_{q,k}(f) \le L,\text{ and } \norm{f}_{L^{\infty}}\le L}, 
\]
for $ 0 < q \le \infty, \text{ and } \gamma, L > 0$.
The best approximant in \eqref{eq:approxErr} exists, but is in general non-unique, see e.g., \citet[Chapter 12]{DeLo93}. It follows readily from definition that $\A_q^{\gamma} = \bigcup_{L > 0} \A_{q,L}^{\gamma}$ and that $\A^\gamma_{q_1,L} \subseteq \A^\gamma_{q_2,L}$ for all $q_1 \ge q_2$. Note that $\A_q^{\gamma}$ is actually an interpolation space between $L^q$ and some Besov space \citep[see][]{Pet88}. The order $\gamma$ of these spaces (or classes) reflects the speed of approximation of $f$ by step functions as the number of change-points increases. It is further known that if $f$ lies in $\A^\gamma_q$ for some $\gamma > 1$ and if $f$ is piecewise continuous, then $f$ is piecewise constant, see \cite{BurHal75} (which is often referred to as a \emph{saturation} result in the approximation theory community). Thus, it {is custom} to consider $\A^{\gamma}_q$ with $0<\gamma\le 1$.

The rates of convergence for approximation classes are provided below.
\begin{theorem}[{Adaptation II}]\label{th:approx}
Let $0 <  p,\, r < \infty$, $\max\{p,2\} \le q \le \infty$, and {assume that $\hat{f}_n$ is a multiscale change-point segmentation estimator in Definition~\ref{def:mrseg}} with constants $c$ and $\delta$, and {universal} threshold {as in~\eqref{eq:defQ}.} Then 
\begin{equation*}
\limsup_{n \to \infty}\,(\log n)^{-\frac{\gamma+(1/2-1/p)_+}{2\gamma+1}}n^{\frac{2\gamma}{2\gamma+1}\min\{1/2,1/p\}}\sup_{f \in \A^{\gamma}_{q,L}}\E[a]{\norm{\hat{f}_n - f}_{L^{p}}^r}^{1/r} < \infty.
\end{equation*}
The same result also holds almost surely if we drop the expectation $\E{\cdot}$.
\end{theorem}

\begin{proof}
See Appendix~\ref{app:approx}.
\end{proof}

\begin{remark}\label{rem:adapt2}
Similar to Theorem~\ref{th:step}, the above theorem shows that {any} multiscale change-point segmentation method with a universal threshold  automatically adapts to the smoothness of the approximation spaces, in the sense that it has a faster rate for larger order $\gamma$.  Note that such convergence rates are minimax optimal (up to a log-factor) over $\A^{\gamma}_{q,L}$ for every $0 < \gamma \le 1$, $2 \le q \le \infty$ and $L > 0$, see Example~\ref{ex:smooth} (i) below. {Also, we point out that Theorem~\ref{th:approx} still holds if one uses the refined rule in~\eqref{eq:refineQ} for the choice of threshold $\thd$, see  Appendix~\ref{app:approx} and also Remark~\ref{rem:adapt}}. 

Moreover, note that the convergence rates of the multiscale change-point segmentation methods above {generalize} the rates reported in~\cite{BoyKemLieMunWit09} for jump-penalized least square estimators, and are faster than the rates  reported in~\cite{Fry07} for the unbalanced Haar wavelets based estimator, with the difference being in log-factors.   
\end{remark}

\begin{example}\label{ex:smooth}
(i) \emph{(Piecewise) H\"{o}lder functions.}
For $0 < \alpha \le 1$ and $L > 0$, we consider the H\"{o}lder function classes
\begin{multline*}
H^{\alpha}_L \equiv H^{\alpha}([0,1))\coloneqq \set[b]{f \in \D}{\norm{f}_{L^{\infty}} \le L,  \text{ and } \\
 \abs{f(x_1) - f(x_2)} \le L \abs{x_1 - x_2}^{\alpha} \text{ for all } x_1,x_2 \in [0,1)}, 
\end{multline*}
and the piecewise H\"{o}lder function classes with at most $\kappa$ jumps
\begin{multline*}
H^{\alpha}_{\kappa,L} \equiv H^{\alpha}_{\kappa,L}([0,1)) \coloneqq \set[B]{f \in \D}{\text{there is a partition }\{I_i\}_{i = 0}^l, \text{ with } l\le \kappa,\text{ of }[0,1) \\ \text{ such that } f\big\vert_{I_i} \in H^{\alpha}_L(I_i) \text{ for all possible } i}. 
\end{multline*}
Obviously, the latter one contains the former as a special case when $\kappa = 0$, that is,  $H^{\alpha}_{0,L} \equiv H^{\alpha}_{L}$. 
It {is easy to see} that $H^{\alpha}_L \subseteq \A^{\alpha}_{q,L'}$ with $L' \ge L$, $0< q \le \infty$, and $H^{\alpha}_{\kappa, L} \subseteq \A^{\alpha}_{q,L'}$ with $L' \ge L(\kappa+1)^{\alpha+1/2}$, $0< q \le \infty$. 

It is known that the fastest possible rate over $H^{\alpha}_L$, $0 < \alpha \le 1$, is at most of order $n^{-2\alpha/(2\alpha+1)\min\{1/2,1/p\}}$ with respect to the $L^p$-loss, $0 <  p < \infty$, see e.g.~\citep{IbrHas81}. Thus, as a consequence of Theorem~\ref{th:approx}, the multiscale change-point segmentation method with a universal threshold is simultaneously minimax optimal (up to a log-factor) over $\A_{q,L}^\alpha$, $H^{\alpha}_L$ and $H^{\alpha}_{\kappa, L}$ for every $\kappa \in \N_0$, $\max\{p,2\}\le q\le \infty$, $0 < \alpha \le 1$ and $L > 0$, that is, adaptive to the smoothness order $\alpha$ of the underlying function. 

\medskip

(ii) \emph{Bounded variation functions.}
Recall that the (total) variation $\norm{\cdot}_{\TV}$ of a function $f$ is defined as
\[
\norm{f}_{\TV} \coloneqq \sup\set[B]{\sum_{i = 0}^m \abs{f(x_{i+1}) - f(x_i)}}{0 = x_0 < \cdots < x_{m+1} =1, \, m\in \N}.
\]
We introduce the {c\`{a}dl\`{a}g} bounded variation classes 
\[
\BV_L \equiv \BV_L([0,1)) \coloneqq \set[b]{f \in \D}{\norm{f}_{L^\infty} \le L, \text{ and } \norm{f}_{\TV} \le L} \quad \text{ for } L > 0. 
\]
Elementary calculation, together with Jordan decomposition, implies that
$$
\BV_L \subseteq \A^1_{q, L'}\qquad\text{ for } L' \ge L \text{ and } 0< q\le \infty.
$$ 
Since the H\"{o}lder class $H^1_{L} \subseteq \BV_{L}$, the best possible rate for $ \BV_{L}$ {cannot be} faster than that for  $H^1_{L}$, which is of order $n^{-2/3\min\{1/2,1/p\}}$. Then, Theorem~\ref{th:approx} implies that the multiscale change-point segmentation method attains the minimax optimal rate (up to a log-factor) over the bounded variation classes $\BV_{L}$ for $L > 0$. 
\end{example}

All the examples above concern functions of smoothness order $\le 1$. For smoother functions, say $H^\alpha_L$ with $\alpha > 1$~\citep[see e.g.][for definition]{Tsy09}, it holds that  $H^\alpha_L \subseteq \A_q^1$ but $H^\alpha_L \not\subseteq \A^\gamma_q$ for any $\gamma>1$. Thus, by Theorem~\ref{th:approx}, we obtain that multiscale change-point segmentation estimators attain (up to a log-factor)  the rates of order $n^{-2/3\min\{1/2,1/p\}}$ for $H_L^\alpha$ with $\alpha > 1$ in terms of $L^p$-loss. Note that such rates are suboptimal, but turn out to be the saturation barrier for every piecewise constant segmentation estimator; As we will see in Example~\ref{ex:sat} in Section~\ref{ss:oSeg}, piecewise constant segmentation estimators even with the oracle choice of change-points cannot attain faster rates for functions of smoothness order $>1$. 

In summary, we {find} that the multiscale change-point segmentation methods with universal parameter choice~\eqref{eq:defQ} or refined choice~\eqref{eq:refineQ} are minimax optimal {(up to log factors)} simultaneously over sequences of step function classes $\mathcal{S}_L(k_n)$ ($k_n = o(n), \, L > 0$), and over approximation spaces $\A^{\gamma}_{q,L}$ ($0< \gamma \le 1,\, 2\le q\le \infty, L >0$). This in particular includes  sequences of step function classes $\mathcal{S}_L(n^{\theta})$ ($0 \le \theta < 1, \, L > 0$), H\"{o}lder classes $H^{\alpha}_{L}$ and $H^{\alpha}_{\kappa, L}$ ($0< \alpha \le 1, \, \kappa \in \N_0,\, L >0$), and bounded variation classes $\BV_L$ ($L > 0$). 

\section{Feature detection}\label{s:feature}

The convergence rates {in Theorems~\ref{th:step} and~\ref{th:approx}} not only reflect the average performance in recovering the truth over its domain, but also, as a byproduct, lead to further statistical justifications on detection of features, such as change-points, modes and troughs.

\begin{proposition}\label{pp:feature}
Assume model~\eqref{eq:model} and let the truth $f \equiv f_{k_n}\in\mathcal{S}_L(k_n)$ be a sequence of step functions with up to $k_n$ jumps. By $\Delta_n$ and $\lambda_n$ denote the smallest jump size, and the smallest segment length of $f_{k_n}$, respectively.
Let $\hat{f}_n$ be a multiscale change-point segmentation method in Definition~\ref{def:mrseg} with constants $c$, $\delta$, interval system $\mathcal{I}$, and universal threshold $\thd$ in~\eqref{eq:defQ} or~\eqref{eq:refineQ}.  If $\lim_{n\to \infty}{k_n\log n}/({\lambda_n \Delta_n^2n}) = 0, $ then there is a constant $C$ depending only on $c$, $\thd$ and $k_n$ such that 
\[
\lim_{n\to \infty}\Prob[B]{\#J(\hat{f}_n) = \#J(f_{k_n}),\, d\bigl(J(\hat{f}_n); J(f_{k_n})\bigr) \le C \frac{k_n\log n}{\Delta_n^2 n} } = 1,
\]
with $d\bigl(J(\hat{f}_n); J(f_{k_n})\bigr) \coloneqq \max_{\tau \in J(f_{k_n})}\min_{\hat{\tau} \in J(\hat{f}_n)} \abs{\tau -\hat\tau}$.
\end{proposition}

\begin{proof}
By Theorem~\ref{th:step} and~\citet[Theorem 8]{LSRT16} it holds almost surely that $d\bigl(J(\hat{f}_n); J(f_{k_n})\bigr) \le C_1 {k_n\log n}/(\Delta_n^2 n)$, and thus $\Prob[b]{\#J(\hat{f}_n) \ge \#J(f_{k_n})} \to 1$. This, together with the fact that $\Prob[b]{\#J(\hat f_n) > \#J(f_{k_n})} \le \mathcal{O}(n^{-r}) \to 0$, completes the proof. 
\end{proof}

\begin{remark}
Proposition~\ref{pp:feature} concerns step functions, and is a typical consistency result in change-point literature {\citep[e.g.][]{BoyKemLieMunWit09,HarLev10,ChCh17}}. It in particular applies to SMUCE~\citep{FriMunSie14} and FDRSeg~\citep{LMS16}, where the same error rate on the accuracy of estimated change-points is reported, and is of the fastest order known up to now~\citep[see also][]{Fry14}. 
\end{remark}

Assume now $f \in \mathcal{D}$, an arbitrary (not necessarily piecewise constant) function. We consider a similar concept of change-points as for step functions. To this end, we define, for any $\varepsilon > 0$,  the jump locations of~$f$ as $J_{\varepsilon}(f)\coloneqq\set{x}{\abs{f(x) - f(x-0)} > \varepsilon}$, and the jump sizes as $\Delta_f^\varepsilon \coloneqq \min\{\abs{f(x) - f(x-0)}\,\colon\,x \in J_\varepsilon(f)\}$. Note that they are well-defined, since it follows from \citet[Lemma 1 in Section 12]{Bil99} that $\#J_{\varepsilon}(f) < \infty$ and $\Delta_f^\varepsilon> 0$. In addition, we introduce the local mean of $f$  over an interval $I$ as $m_I(f)\coloneqq \int_{I}f(x)dx/\abs{I}$. Such local means $m_I(f)$ on different intervals $I$ actually shed light on the shape of $f$, such as pieces of increases and decreases, and thus modes and troughs.

\begin{theorem}\label{th:feature}
Assume model~\eqref{eq:model}, and the truth $f\in\D$.  Let $\hat{f}_n$ be a multiscale change-point segmentation method in Definition~\ref{def:mrseg} with constants $c$, $\delta$, and interval system $\mathcal{I}$. 
\begin{itemize}
\item[\emph{(i)}] 
If $f \in \A^{\gamma}_{2,L}$ with $\gamma, L > 0$, and the threshold $\thd$ of $\hat f_n$ is chosen as in~\eqref{eq:defQ} or~\eqref{eq:refineQ}, then
\begin{align}
&\lim_{n\to \infty}\Prob[a]{\#\mathrm{modes}(\hat f_n) \ge \#\mathrm{modes}(f);\,\#\mathrm{troughs}(\hat f_n) \ge \#\mathrm{troughs}(f)} = 1, \label{eq:mt} \\
\text{and}\quad& \lim_{n\to\infty}\Prob[g]{d\bigl(J(\hat f_n), J_\varepsilon(f)\bigr) \le \frac{C}{(\Delta^\varepsilon_f)^2}\Bigl(\frac{\log n}{n}\Bigr)^{\frac{2\gamma}{2\gamma+1}};\,\#J(\hat{f}_n) \ge \#J_\varepsilon(f)}=1,\label{eq:jmp} 
\end{align}
where $C$ is a constant depending only on $\eta$, $c$ and $L$. 
\item[\emph{(ii)}]
If the threshold $\thd$ of $\hat f_n$ is chosen as $\thd = \thd(\beta)$, then it holds with probability at least $(1-\beta)$ that for any $I_1, I_2\in \mathcal{I}$, where $\hat f_n$ is constant, 
\begin{equation}\label{eq:sft}
m_{I_1}(\hat f_n) > m_{I_2}(\hat f_n) + r_{I_1} + r_{I_2} \qquad \text{with } r_I = \frac{2\bigl(\thd(\beta)+s_{I}\bigr)}{\sqrt{n\abs{I}}}, 
\end{equation}
implies  $m_{I_1}(f) > m_{I_2}(f)$, simultaneously over all such pairs of $I_1$ and $I_2$. 
\end{itemize}
\end{theorem}
\begin{proof}
See Appendix~\ref{app:feature}. 
\end{proof}
\begin{remark}\label{rem:ft}
Since step functions lie in $\mathcal{A}_{2}^\gamma$ for all $\gamma >0$, assertion~\eqref{eq:jmp} ``formally'' reproduces  Proposition~\ref{pp:feature} partially for the case that the step function $f$ is fixed, by letting $\gamma$ tend to infinity. Moreover, the statistical justifications of Theorem~\ref{th:feature}~(i) are of one-sided nature. Note that statistical guarantees for the reverse order are in general not possible, as long as an arbitrary number of jumps / features on small scales cannot be excluded, see e.g.~\cite{Don88}. However, multiscale change-point segmentation methods will not include too many artificial features (e.g., jumps, modes or troughs), due to their parsimony nature by construction, namely, minimization of the number of jumps, see~\eqref{eq:mr_segment}. 

More importantly, Theorem~\ref{th:feature}~(ii) states that large increases (or decreases) of multiscale change-point segmentation estimators imply increases (or decreases) of the true signal. This is actually a finite-sample inference guarantee, and holds simultaneously for many intervals, which thus provides inference guarantee on modes and troughs. In this way, we can discern a collection of genuine features among all the detected features, with controllable confidence. 
To be precise, let  
$\hat f_n = \sum_{i = 1}^{\hat k} \hat c_i \Ind_{[\hat \tau_{i-1},\,\hat \tau_i)} \text{ with }0 = \hat\tau_0 < \cdots < \hat \tau_{\hat k} = 1\text{ and }\hat c_i \neq \hat c_{i+1}$ 
be a multiscale change-point segmentation estimator with threshold $\eta(\beta)$. 
\begin{enumerate}[(i)]
\item{\emph{Increase or decrease.}}
Let $\hat\tau_{i+1/2} = (\hat\tau_i+\hat\tau_{i+1})/2$. Define
\begin{align*}
&u^R_i = \min_{I \in \I,\,  I \subseteq [\hat\tau_i, \,\hat\tau_{i+1/2})} (\hat c_{i+1} + r_I), &&l^R_i = \max_{I \in \I,\,  I \subseteq [\hat\tau_i, \,\hat\tau_{i+1/2})} (\hat c_{i+1} - r_I),\\
\text{and } \quad & u^L_i = \min_{I \in \I,\,  I \subseteq [\hat\tau_{i-1/2}, \,\hat\tau_{i})} (\hat c_i + r_I), &&l^L_i = \max_{I \in \I,\,  I \subseteq [\hat\tau_{i-1/2}, \,\hat\tau_{i})} (\hat c_i - r_I).
\end{align*}
Then, by Theorem~\ref{th:feature}~(ii), there is at least an increase (or a decrease) of $f$ on interval $[\hat\tau_{i-1/2},\,\hat\tau_{i+1/2})$ if $u_i^L < l_i^R$ (or if  $l_i^L > u_i^R$) with confidence level no less than ($1-\beta$). Further, because of the simultaneous confidence control, the inferred increases and decreases on non-overlapped intervals $[\hat\tau_{i-1/2},\hat\tau_{i+1/2})$ leads naturally to inference on modes and troughs. 
\item{\emph{Change-point.}}
Let $m \ll n$. 
Consider intervals 
$$
[\hat c_{i} - r_{[\hat\tau_i - m/n, \, \hat \tau_i)},\, \hat c_{i} + r_{[\hat\tau_i - m/n, \, \hat \tau_i)}] \text { and }
[\hat c_{i+1} - r_{[\hat\tau_i , \, \hat \tau_i+m/n)},\, \hat c_{i+1} + r_{[\hat\tau_i, \, \hat \tau_i +m/n)}]\,. 
$$
If both intervals are disjoint, we call $\hat\tau_i$ a significant change-point. Strictly speaking, in such a case, it follows from Theorem~\ref{th:feature}~(ii) only that 
\begin{equation}\label{eq:cpt}
m_{[\hat\tau_i - m/n, \, \hat \tau_i)}(f) \neq m_{[\hat\tau_i , \, \hat \tau_i+m/n)} (f)
\end{equation}
with confidence level at least ($1-\beta$). However, for a fixed $f$, when $n$ is large enough, \eqref{eq:cpt} will imply $f(\tau_i) = f(\tau_i + 0) \neq f(\tau_i - 0)$, i.e., a jump of $f$ at $\tau_i$, for some $\tau_i$. Thus, a significant change-point in most cases leads to a true change-point. In practice, we recommend $m = \lfloor \log n\rfloor$ as the default choice. 
\end{enumerate}
See Figure~\ref{fig:feature} (in Section~\ref{s:intro}) for an illustration. The SMUCE has detected $3$ change-points, $1$ mode and $1$ trough. By the method described above, we can claim that the truth has at least $1$ mode (in region $[0.36, \,0.88)$), $1$ trough (in region $[0.1,\,0.63)$) and $2$ change-points (at $0.5$ and $0.75$), with probability at least $90\%$. Such inference is nicely confirmed by the underlying truth.
\end{remark}

\section{Oracle properties}\label{s:oracle}

This section focuses on the oracle properties of multiscale change-point segmentation methods. For simplicity, we restrict ourselves to $\A^\gamma_{2}$ and $L^2$-topology.  

\subsection{Oracle segmentation}\label{ss:oSeg} 
It is well-known that the crucial difficulty in change-point segmentation problems is to infer the locations of change-points; Once the change-point locations are detected, the height of each segment can easily be determined via any reasonable estimator, e.g.~{a} maximum likelihood estimator, locally on each segment \citep[see e.g.][]{KillFeaEck12,Fry14}. In line of this thought, we define 
$$
\Pi_n \coloneqq \set{(\tau_0,\tau_1,\ldots,\tau_{k})}{\tau_0 = 0 < \tau_1< \cdots < \tau_k =1,\, k\in\N,\text{ and } \{n\tau_i\}_{i = 1}^k \subseteq \N}.
$$ 
For each $\tau \equiv (\tau_0,\ldots,\tau_k)$, we introduce the piecewise constant segmentation estimator~$\hat f_{\tau,n}$, conditioned on $\tau$, for model~\eqref{eq:model} as 
$$
\hat f_{\tau,n} \coloneqq \sum_{i =1}^k \hat c_i \Ind_{[\tau_{i-1},\tau_i)} \qquad \text{with }\hat c_i \coloneqq \frac{\sum_{j\in [n\tau_{i-1}, n\tau_i)} y^n_j}{n(\tau_i - \tau_{i-1})}.
$$
\begin{theorem}\label{th:oSeg}
Assume model~\eqref{eq:model}, and sub-Gaussian noises s.t.~$\E{(\xi_i^n)^2}\asymp \sigma_0^2$, i.e., for some constants $c_1,c_2$ it holds that $c_1 \sigma_0^2\le\E{(\xi_i^n)^2}\le c_2 \sigma_0^2$ for every possible $i$ and $n$.  Let $\hat{f}_n$ be a multiscale change-point segmentation method in Definition~\ref{def:mrseg} with threshold as in~\eqref{eq:defQ} or~\eqref{eq:refineQ}. Then, there is a {universal} constant $C$ such that for every $f$ in $\cup_{\gamma > 0}\A^\gamma_2 \cap L^\infty$
\[
\E{\norm{\hat f_n -f}^2_{L^2}} \le C \log n \inf_{\tau \in \Pi_n} \E{\norm{\hat f_{\tau,n} - f}_{L^2}^2}\qquad \text{for sufficiently large } n.
\]
\end{theorem}

\begin{proof}
See Appendix~\ref{app:oSeg}. 
\end{proof}

\begin{remark}
Theorem~\ref{th:oSeg} states that multiscale change-point segmentation methods perform nearly (up to a log-factor) as well as the piecewise constant segmentation estimator using an oracle for the change-point locations.  
\end{remark}

We next consider the \emph{saturation phenomenon} of piecewise constant segmentation estimators via a simple example. 

\begin{example}\label{ex:sat}
Assume model \eqref{eq:model} with the truth $f(x) \equiv x$ and the noise $\xi_i^n$ being standard Gaussian. For simplicity, let $n = 6m^3$ with $m \in \N$.  Elementary calculation shows that 
$$
 \E{\norm{\hat f_{\tau_*,n} - f}_{L^2}^2} = \inf_{\tau \in \Pi_n} \E{\norm{\hat f_{\tau,n} - f}_{L^2}^2} = \frac{6^{2/3} + 6^{-1/3}}{12}n^{-2/3}
$$
and $\tau_* = \bigl(0, 1/m, \ldots, (m-1)/m, 1\bigr)$. Note that $f(x)\equiv x$ lies in every H\"older class $H^\alpha_L$ with $0< \alpha <\infty$ and $L\ge 1$, and that the minimax optimal rates in terms of squared $L^2$-risk for $H^\alpha_L$ is of order $n^{-2\alpha/(2\alpha + 1)}$. Thus, it indicates that the piecewise segmentation estimator even with the oracle choice of change-points saturates at smoothness order $\alpha =1$. This in turn explains why multiscale change-point segmentation methods cannot achieve faster rates for functions of smoothness order $\ge1$.
\end{example}

Note that such saturation phenomenon for piecewise constant segmentation estimators is by no means due to the discontinuity of the estimator. In fact, one could discretize a smooth estimator \citep[i.e., wavelet shrinkage estimators,][]{DonJRSS95} on the sample grids $\{i/n\}_{i = 0}^n$ into a piecewise constant one: the discretized version performs equally well as the original estimator in asymptotical sense, since the discretization error vanishes faster than statistical estimation error.  In contrast, the underlying reason for the aforementioned saturation is because piecewise constant segmentation estimators aim to segment data into constant pieces, rather than approximate the truth as well as possible. The purpose of segmentation into constant pieces provides an easy interpretation of the data, but it turns out to be less sufficient if the complete recovery of the truth is the statistical task. To overcome this saturation barrier, one could smoothen each segment based on detected change-point locations~\citep[see][]{BKS71}, which is, however, beyond the scope of this paper.

\subsection{Oracle approximant}\label{ss:oApp}

Here we examine the performance of multiscale change-point segmentation methods $\hat f_n$ by comparing it with the best piecewise constant approximants of $f$ with up to $\# J(\hat f_n)$ jumps. By means of compactness arguments and the convexity of $L^2$-norm, we can define 
\begin{equation}\label{eq:bestApp}
{f^\app_{k}} \in \mathop{\argmin}\limits_{g \in \mathcal{S}, \, \#J(g) \le k} \norm{f - g}_{L^2} \qquad \text{ for } k \in \N,
\end{equation}
which might be non-unique, as mentioned earlier in Section~\ref{ss:approx}. 
\begin{proposition}\label{pp:bestApp}
Assume model~\eqref{eq:model}. Let $\hat{f}_n$ be a multiscale change-point segmentation method in Definition~\ref{def:mrseg} with threshold as in~\eqref{eq:defQ} or~\eqref{eq:refineQ}, and $\hat K_n \coloneqq \# J(\hat f_n)$. Then
\[
\lim_{n\to \infty} \Prob[B]{\sup_{f \in \A_{2, L}^{\gamma}}\norm{f - f^\app_{\hat K_n}}_{L^2} \ge C\sup_{f \in \A_{2,L}^{\gamma}}\norm{f - \hat f_n}_{L^2} }=1\qquad \text{ for some constant } C. 
\]
\end{proposition}
\begin{proof}
Following the proof of Theorem~\ref{th:approx} and Remark~\ref{rem:adapt2}, one can see that 
\begin{equation}\label{eq:prob_to_1}
\lim_{n\to\infty}\Prob[b]{A_n} = 1, 
\end{equation}
where the event  $A_n$ is defined as
\[
A_n \coloneqq \left\{\hat K_n \le k_n, \, \sup_{f \in \A^{\gamma}_{2,L}}\norm{f - \hat f_n}_{L^2}  \le C_2\bigl(\frac{\log n}{n}\bigr)^{\frac{\gamma}{2\gamma+1}}\right\}\quad \text{ with } k_n \coloneqq C_1 \bigl(\frac{n}{\log n}\bigr)^{\frac{1}{2\gamma+1}}. 
\]
{On} the event $A_n$, it holds that 
\begin{equation*}
\sup_{f \in \A^{\gamma}_{2,L}} \norm{f - f^\app_{\hat K_n}}_{L^2} \ge \sup_{f \in \A^{\gamma}_{2,L}} \norm{f - f^\app_{k_n}}_{L^2} \ge C_3 k_n^{-\gamma} \ge C_4 \bigl(\frac{\log n}{n}\bigr)^{\frac{\gamma}{2\gamma+1}} \ge C_5\sup_{f \in \A^{\gamma}_{2,L}} \norm{f - \hat f_n}_{L^2}.
\end{equation*}
This, together with~\eqref{eq:prob_to_1}, concludes the proof.
\end{proof}

\begin{figure}[!h]
\centering
\includegraphics[width=0.9\textwidth,clip]{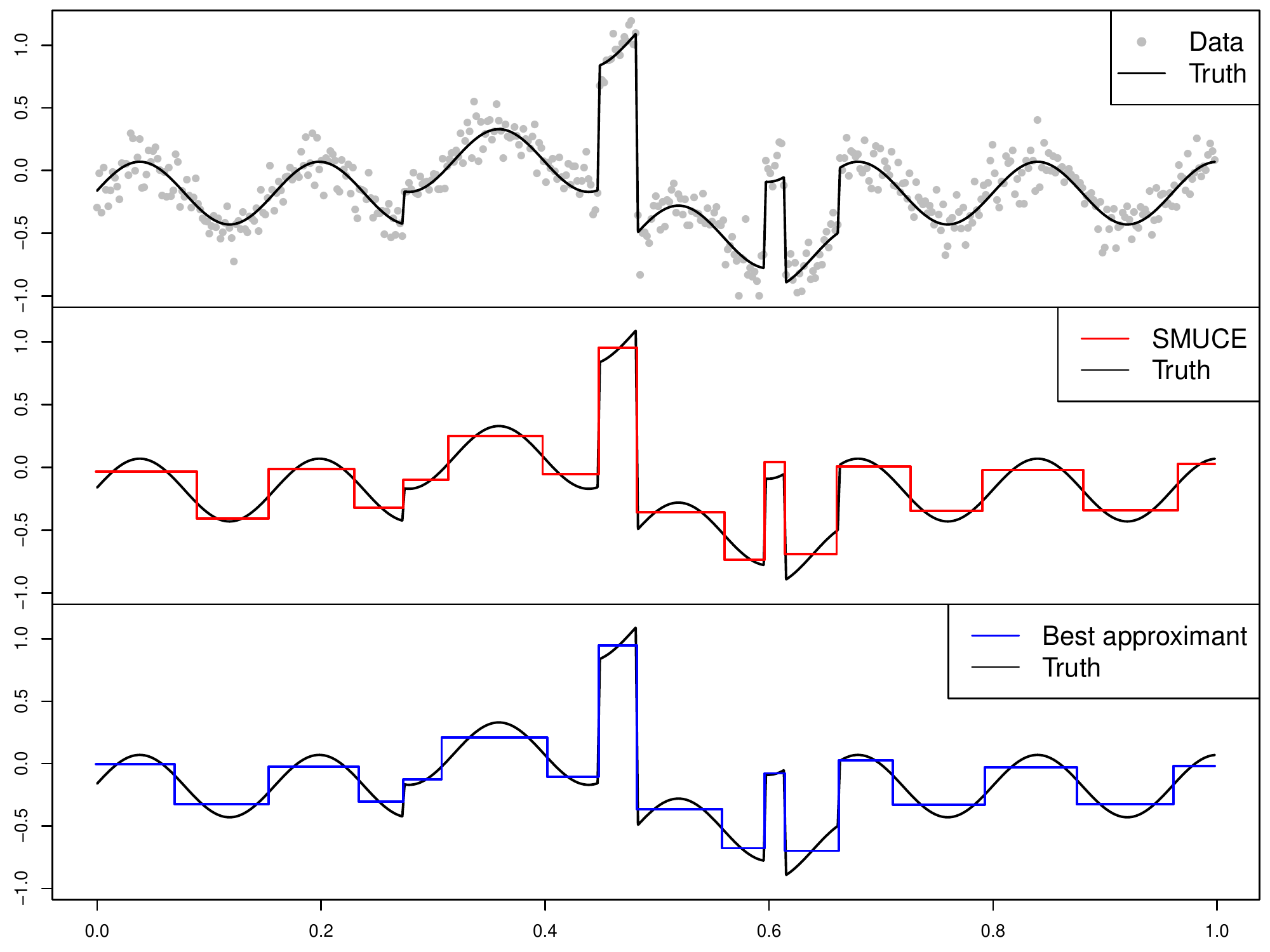}
\caption{{Performance of a particular multiscale change-point segmentation method $\hat f_n$ \citep[SMUCE,][]{FriMunSie14} as oracle approximants for the signal in~\cite{OlsVenLucWig04}, \cite{ZhaSie07}. The bottom panel shows the best approximant $f_{\hat K_n}$, defined in~\eqref{eq:bestApp}, of the truth with up to $\hat K_n$ jumps. Here $\text{SNR} = 3$ and $\norm{f - \hat f_n}_{L^2} = 1.3 \norm{f - f^\app_{\hat K_n}}_{L^2}$.}
\label{oapp_fig}} 
\end{figure}

\begin{remark}
Note that $\norm{f - f^\app_{\hat K_n}}_{L^2} \le \norm{f - \hat f_n}_{L^2}$, and that Proposition~\ref{pp:bestApp}  implies 
\[
\lim_{n \to \infty} \Prob[g]{\sup_{f \in \mathcal{A}^\gamma_{2,L}}\frac{\norm{f - f^\app_{\hat K_n}}_{L^2}}{\norm{f - \hat f_n}_{L^2}} \ge C} = 1.
\]
This indicates that $\hat f_n$ performs {almost (up to a constant) as} well as the best approximants $f^\app_{\hat K_n}$ of $f$ over all step functions with up to $\hat K_n$ jumps, see Figure~\ref{oapp_fig} for a visual illustration.
\end{remark}

\section{Simulation study}\label{s:numerics}

Note that in the definition of multiscale change-point segmentation methods, we consider only the local constraints on the intervals where  candidate functions are constant. This ensures the structure of the corresponding optimization problem~\eqref{eq:mr_segment} to be a directed acyclic graph, which makes \emph{dynamic programming} algorithms~\citep[cf.][]{Bel57} 
applicable to such a problem~{\cite[see also][]{FriKemLieWin08}.} Moreover, the computation can be substantially accelerated by incorporating \emph{pruning} ideas as recently developed in~\cite{KillFeaEck12}, \cite{FriMunSie14} and \cite{LMS16}.  
As a consequence, the {computational} complexity of multiscale change-point segmentation methods can be even \emph{linear} in terms of the number of observations, in case that there are many change-points, see~\cite{FriMunSie14} and~\cite{LMS16} for further details. 

We now investigate the {finite sample} performance of multiscale change-point segmentation methods from {the previously discussed} perspectives. For brevity, we only consider a particular multiscale change-point segmentation method, SMUCE~\citep{FriMunSie14},  and stress that the results are similar for other multiscale change-point segmentation methods {of type~\eqref{eq:mr_segment}} (which are not shown here), see e.g.~\cite{LMS16} for {an extensive} simulation study. For SMUCE, we use the implementation of an efficient pruned dynamic program from the CRAN R-package ``stepR'', select the system of all intervals with dyadic lengths for the multiscale constraint, and choose $\thd(\beta)$ as the threshold, which is simulated by 10,000 Monte-Carlo simulations. In what follows, the noise is assumed to be Gaussian with a known noise level $\sigma$, and SNR denotes the signal-to-noise ratio {$\norm{f}_{L^2}/\sigma$}.

\subsection{Stability}\label{ss:Stability}

We first examine the stability of multiscale change-point segmentation methods with respect to the significance level $\beta$, {i.e.~to the threshold $\thd$}. The test signal $f$~\citep[adopted from][]{OlsVenLucWig04,ZhaSie07} has 6 change points at 138, 225, 242, 299, 308, 332, and its values on each segment are -0.18, 0.08, 1.07, -0.53, 0.16, -0.69, -0.16, respectively.  Figure~\ref{sz0_alpha} presents the behavior of SMUCE  with threshold $\thd = \thd(\beta)$ for different choices of significance level $\beta$. In fact, for the shown data, SMUCE detects the correct number of change-points, and recovers the location and the height of each segment in high accuracy, for the whole range of $0.06 \le \beta \le 0.94$ {(i.e.~$ 0.47 \sqrt{\log n} \ge \thd \ge -0.04\sqrt{\log n}$).} {Only} for smaller $\beta$ ($ < 0.06, {\text{ i.e.~}\thd > 0.47\sqrt{\log n}}$), SMUCE tends to underestimate the number of change-points  (see the second panel of Figure~\ref{sz0_alpha} for example, where the missing change-point is marked by a vertical line), while, for larger $\beta$ ($ > 0.94, {\text{ i.e.~}\thd < -0.04\sqrt{\log n}}$), it is inclined to recover false change points (as shown in the last panel of Figure~\ref{sz0_alpha}). Note that in either case the estimated locations and heights of the remaining segments (away from the missing/spurious jumps) are fairly accurate.  This reveals that SMUCE is remarkably stable with respect to the choice of $\beta$ (or $\thd$), {in accordance with the assumptions~\eqref{eq:defQ} and~\eqref{eq:refineQ} of Theorem~\ref{th:step}, Remark~\ref{rem:adapt} and Proposition~\ref{pp:feature} (i).}  

\begin{figure}[!h]
\centering
\includegraphics[width=0.9\textwidth,clip]{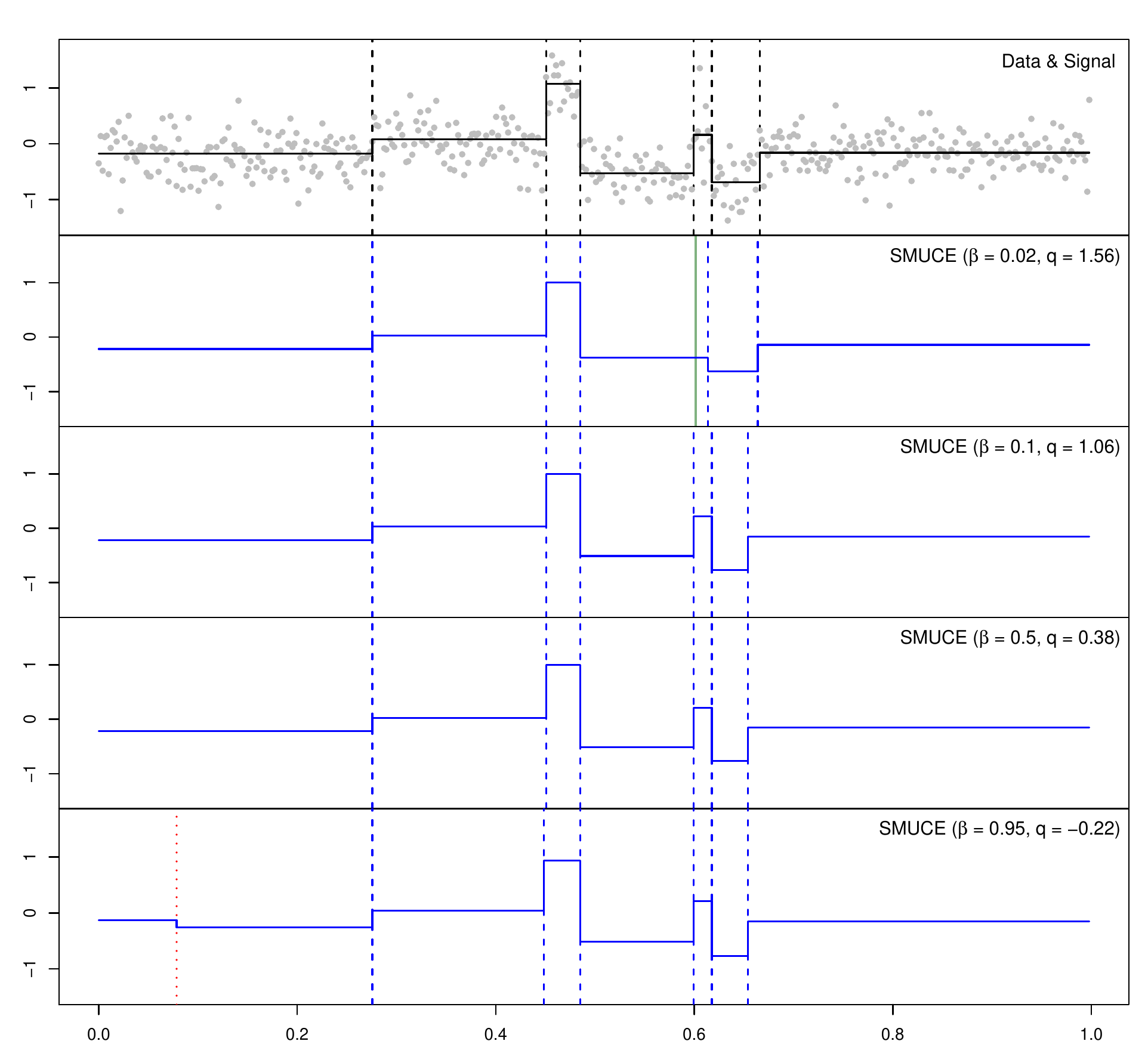}
\caption{Estimation of the step signal in~\cite{OlsVenLucWig04} and \cite{ZhaSie07} by SMUCE with $\thd = \thd(\beta)$ for different $\beta$ (sample size $n = 497$, and $\text{SNR} = 1$).}
\label{sz0_alpha}
\end{figure}

\begin{figure}[!t]
\centering
\includegraphics[width=0.9\textwidth,clip]{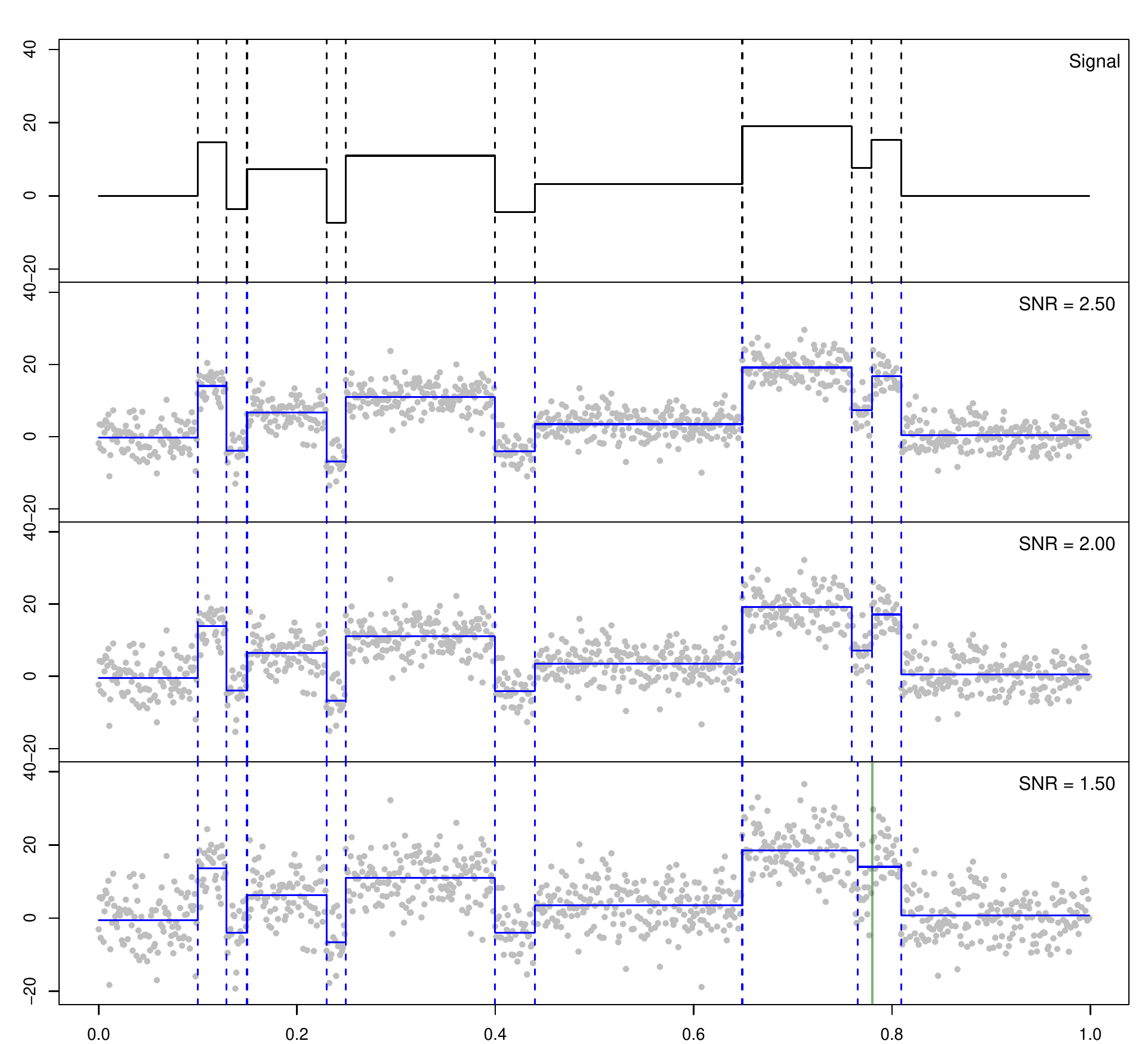}
\caption{Blocks signal: SMUCE for various noise levels (sample size $n = \text{1,023}$).}
\label{blocks_noise}
\end{figure}

\subsection{Different noise {levels}} 

We next investigate the impact of the noise level (or equivalently SNR) on multiscale change-point segmentation methods.  {We consider the recovery of the Blocks signal~\citep{DonJoh94} for different noise levels. } The {results for} SMUCE {at} significance level $\beta = 0.1$ {are} summarized in Figures~\ref{blocks_noise}.  {It shows that SMUCE recovers the signal  well for low and medium noise levels, while misses one or two {small scale} features {for small SNR.} }

\subsection{Robustness}

To study the robustness of multiscale change-point segmentation methods in case of model misspecification, we introduce a local trend component as in~\cite{OlsVenLucWig04} and~\cite{ZhaSie07} to the test signal $f$ in Section~\ref{ss:Stability}, which leads to the model 
\begin{equation}\label{eq:sz_signal}
y_{i}^{n} = \left(\bar f_i^n +0.25b\sin(a\pi i)\right)+\xi_{i}^n,\qquad i = 0, \ldots, n-1.
\end{equation}

Weak background waves: We simulate data {for} $a=0.025$ and $b=0.3$, and apply SMUCE again with various choices of $\beta$, see Figure~\ref{sz2_alpha}. In accordance with the previous simulations and Proposition~\ref{pp:feature} (ii), SMUCE captures all relevant features of the signal {again} for a wide range of $\beta$ ($0.08 \leq \beta \leq 0.29$).

Strong background waves: When the parameter $b$ becomes larger, i.e., the fluctuation is {stronger}, SMUCE  captures the fluctuation by inducing additional change-points according to Theorems~\ref{th:approx},~\ref{th:oSeg} and Proposition~\ref{pp:bestApp}. Figure~\ref{sz_trendnoise} illustrates the performance of SMUCE for the signal in~\eqref{eq:sz_signal} with $b=1.0$ and $b=1.2$ under different noise levels. {With high probability (see Section~\ref{s:feature}) it recovers the pieces of increases and decreases and hence the relevant modes and troughs. }

\begin{figure}[!t]
\centering
\includegraphics[width=0.9\textwidth,clip]{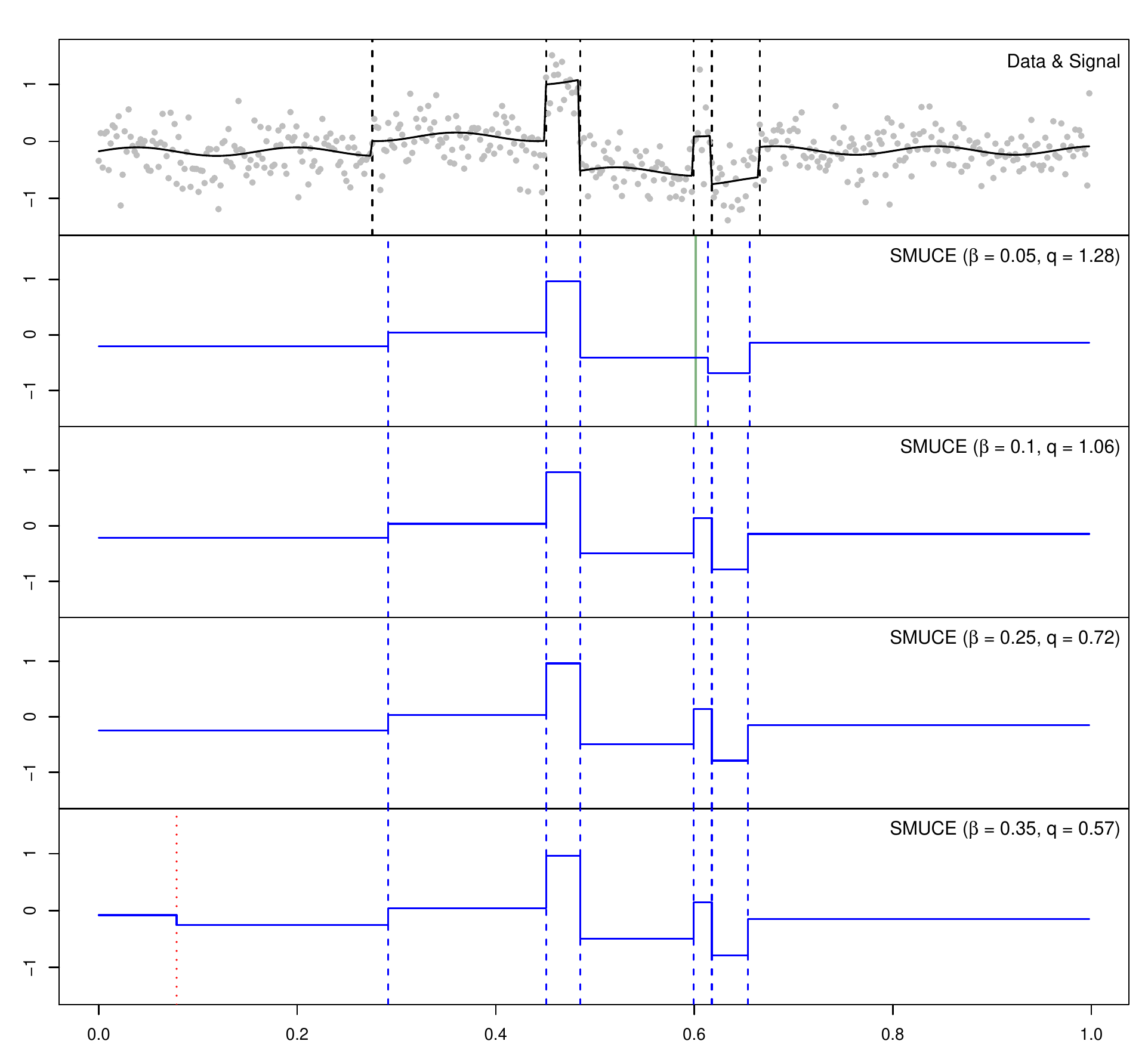}
\caption{Estimation of the signal in~\eqref{eq:sz_signal} {($a=0.025, b=0.3$)} by SMUCE with $\thd = \thd(\beta)$ for different $\beta$ (sample size $n = 497$, and $\text{SNR} = 1$).}
\label{sz2_alpha}
\end{figure}

\begin{figure}[!t]
\centering
\includegraphics[width=0.9\textwidth,clip]{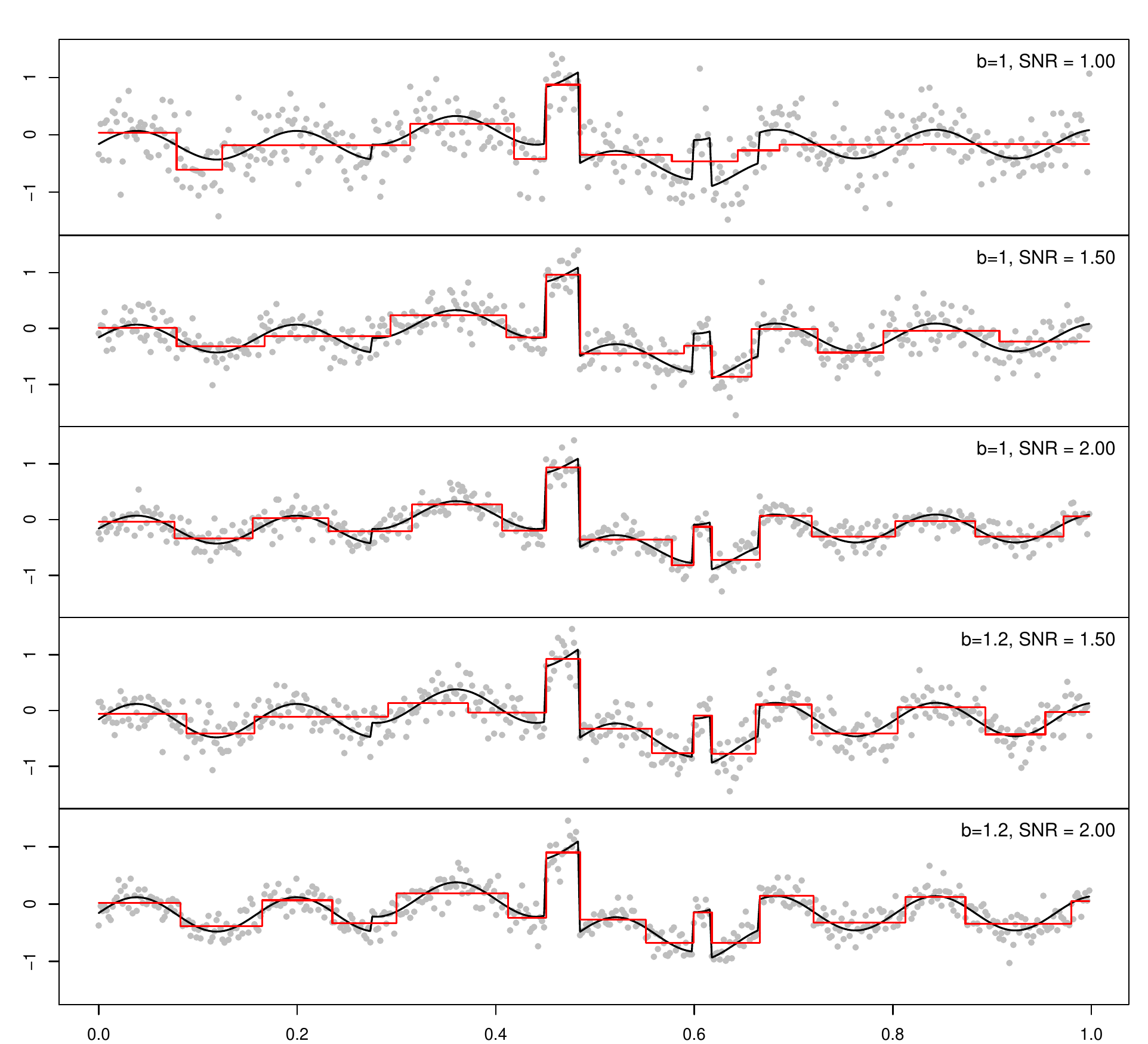}
\caption{Estimation of the signal in~\eqref{eq:sz_signal} with $a = 0.025$ and $b = 1$ or $1.2$ by SMUCE for various noise levels (sample size $n = 497$). } \label{sz_trendnoise}
\end{figure}

\subsection{Empirical convergence rates}

 Lastly, we empirically explore {how well the finite sample risk is approximated by our asymptotic approximations.} The test signals are Blocks and Heavisine~\citep{DonJoh94}. In Figure~\ref{conver1}, we {display} the average of $L^2$-loss of SMUCE with significance level $\beta = 0.1$ over 20 repetitions for a range of sample sizes from 1,023 to 10,230. From Figure~\ref{conver1} we draw that the empirical convergence rates are quite close to the minimax optimal rates (indicated by slopes of the red straight lines), which confirms our theoretical findings in Theorems~\ref{th:step} and~\ref{th:approx}.

\begin{figure}[!t]
\centering
\includegraphics[width=0.9\textwidth,clip]{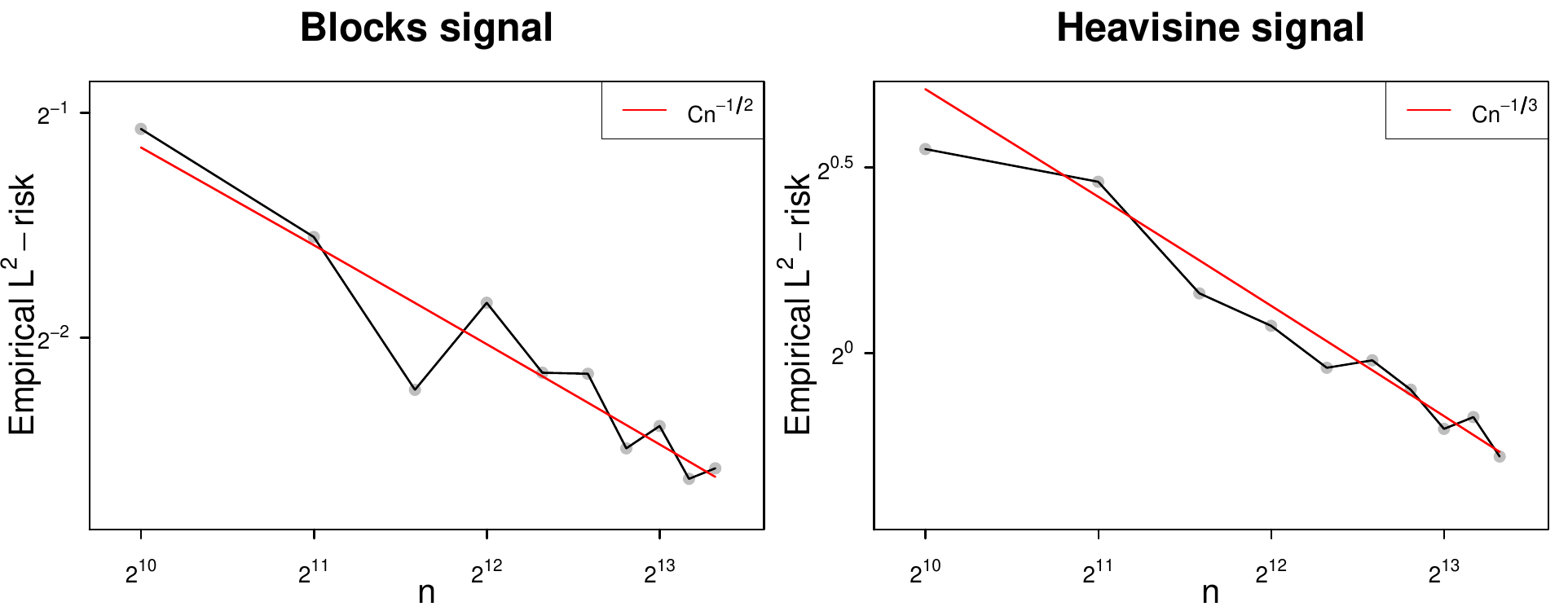}
\caption{Convergence rates of SMUCE for Blocks and Heavisine signals ($\text{SNR} = 2.5$).}
\label{conver1}
\end{figure}

\section{Conclusion}\label{s:discuss}
In this paper we focus on the convergence analysis for multiscale change-point segmentation methods, a general family of change-point estimators based on the combination of variational estimation and multiple testing over different scales, in a nonparametric regression setting with special emphasis on step functions while allowing for various distortions. We found that the estimation difficulty for a step function is mainly determined by its number of jumps, and shown that multiscale change-point segmentation methods attain the nearly optimal convergence rates for step functions with asymptotically bounded or even increasing number of jumps. As a robustness study, we also examined the convergence behavior of these methods for more general functions, which are viewed as distorted jump functions. Such distortion is precisely characterized by certain approximation spaces. In particular, we have derived  nearly optimal convergence rates for multiscale change-point segmentation methods in case that the regression function is either a (piecewise) H\"older function or a bounded variation function.  {Remarkably}, these methods automatically adapt to the unknown smoothness for all aforementioned function classes, as the only tuning parameter can be selected in a universal way.  The convergence rates also provide statistical justification with respect to the detection of features, such as change-points and modes (or troughs).  In addition, the multiscale change-point segmentation methods $\hat f_n$ are shown perform nearly as well as the oracle piecewise constant segmentation estimators, and the best piecewise constant {(oracle)} approximants of the truth with less or the same number of jumps as $\hat f_n$.

The multiscale change-point segmentation methods, however, cannot attain faster rates for functions of stronger smoothness than above, which is indeed a common saturation shared by all piecewise constant segment estimators.  This can be improved by considering piecewise polynomial estimators~\citep[see e.g.][]{Spo98}, but the proper combination with multiscale methodology needs further investigation~\citep[see the rejoinder by][for a first attempt]{FriMunSie14}. Alternatively, certain smoothness penalty can be selected instead of the number of jumps in the formulation of multiscale change-point segmentation, see e.g.~\cite{GrLiMu15}, where the nearly optimal rates are shown for higher order Sobolev/Besov classes. In addition, extension of our results to models with general errors beyond sub-Gaussian, such as {heavy tailed distributions}, {and stationery Gaussian processes~\citep[see e.g.][]{ChSch15},} would be interesting for future research.

\appendix

\setstretch{1.0}

\section{Proofs} 

\subsection{Proof of Theorem~\ref{th:step}}\label{app:step}
We first consider Part (i), and structure the proof into three steps. 
\newline 
\noindent
a) \emph{Good noise case}. Assume that the truth $f$ lies in the multiscale constraint, 
$$T_{\mathcal{I}}(y^n; f) \le a \sqrt{\log n}.$$
In particular, $T_{\mathcal{I}}(y^n; f) \le a \sqrt{\log n}$, so $\#J(\hat{f}_n) \le \# J(f) \le k_n$. Let intervals $\{I_i\}_{i = 0}^m$ be the partition of $[0,1)$ by $J(\hat{f}_n)\cup J(f)$ with $m \le 2 k_n$. Then
\[
\norm{\hat{f}_n - f}_{L^p}^p = \sum_{i = 0}^m \abs{\hat\theta_i - \theta_i}^p \abs{I_i}\qquad \text{ with } \hat{f}_n|_{I_i} \equiv \hat\theta_i \text{ and } {f}|_{I_i} \equiv \theta_i.
\]
If $\abs{I_i} > c/n$, then by $c$-normality of $\mathcal{I}$, there is $\tilde{I}_i \in \mathcal{I}$ such that $\tilde{I}_i \subseteq I_i$ and $\abs{\tilde{I}_i} \ge \abs{I_i}/c$. It follows that
\[
{\abs[b]{\tilde{I}_i}}^{1/2}\abs[g]{\theta - \frac{1}{n\abs{\tilde{I}_i}}\sum_{j/n \in \tilde{I}_i}y_j^n} \le (a+\delta)\sqrt{\frac{\log n}{n}}\qquad\text{ for } \theta = \theta_i \text{ or }\hat\theta_i,
\]
which, together with $\abs{\tilde{I}_i} \ge \abs{I_i}/c$, implies
\[
\abs{I_i}^{1/2}\abs{\hat\theta_i -\theta_i} \le 2(a+\delta)\sqrt{\frac{c\log n}{n}}.
\]
If $\abs{I_i} \le c/n$, then we have for some $i_0$
\[
\abs{\hat\theta_i-\theta_i} \le \abs{\hat\theta_i - y_{i_0}^n} + \abs{y_{i_0}^n -\bar f_{i_0}^{n}} + 2\norm{f}_{L^\infty} \le 2(a+\delta)\sqrt{{\log n}} + 2L. 
\]
Thus, by combining these two situations, we obtain that
\begin{multline*}
\norm{\hat{f}_n - f}_{L^p}^p \le \sum_{i:\abs{I_i} > c/n}\abs{I_i}  \biggl(2(a+\delta)\sqrt{\frac{c\log n}{n\abs{I_i}}}\biggr)^p
+ \sum_{i:\abs{I_i} \le c/n}\frac{c}{n}\left(2(a+\delta)\sqrt{\log n}+2L\right)^p. 
\end{multline*}
Note that for $0 <  p < 2$, by the H\"{o}lder's inequality, 
\begin{align*}
\sum_{i:\abs{I_i} > c/n} \abs{I_i} \biggl(2(a+\delta)\sqrt{\frac{c\log n}{n\abs{I_i}}}\biggr)^p
\le &\Bigl(\sum_{i:\abs{I_i} > c/n} \abs{I_i}\Bigr)^{1-p/2}\biggl(\sum_{i:\abs{I_i} > c/n} 4(a+\delta)^2\frac{c\log n}{n}\biggr)^{p/2} \\
 \le & \left(4(2k_n+1)(a+\delta)^2\frac{c\log n}{n}\right)^{p/2},
\end{align*}
and for $2 \le p < \infty$,
\begin{align*}
\sum_{i:\abs{I_i} > c/n} \abs{I_i} \left(2(a+\delta)\sqrt{\frac{c\log n}{n\abs{I_i}}}\right)^p
\le & \sum_{i:\abs{I_i} > c/n} \left(2(a+\delta)\sqrt{\frac{c\log n}{n}}\right)^p\left(\frac{c}{n}\right)^{1-p/2}  \\
\le & \frac{(2k_n+1)c}{n}\bigl(4(a+\delta)^2{\log n}\bigr)^{p/2}. 
\end{align*}
Since $k_n = o(n)$, we have as $n \to \infty$,
\begin{equation}\label{eq:goodCaseStep}
\norm{\hat{f}_n-f}_{L^p}^r \le 2^{r/p}\Bigl(\frac{(2k_n+1)c}{n}\Bigr)^{\min\{r/2,r/p\}}\bigl(4(a+\delta)^2{\log n}\bigr)^{r/2}\bigl(1+o(1)\bigr). 
\end{equation}
\newline 
\noindent
b) \emph{Almost sure convergence}.
For each $I\in\mathcal{I}$, note that $(n\abs{I})^{-1/2}\sum_{i/n \in I} \xi_i^n$ is again sub-Gaussian with scale parameter $\sigma$, so $\Prob[b]{(n\abs{I})^{-1/2}\abs{\sum_{i/n \in I} \xi_i^n} > x} \le 2 \exp(-x^2/2\sigma^2)$ for any $x > 0$. 
Then, by Boole's inequality, it holds that 
\begin{equation} \label{eq:tailMRstat}
\begin{aligned} 
\Prob[a]{T_{\mathcal{I}}(y^n; f) > a\sqrt{\log n}} 
\le & \Prob[a]{\sup_{I \in \mathcal{I}}\frac{1}{\sqrt{n\abs{I}}}\abs[B]{\sum_{i/n \in I} \xi_i^n} > (a-\delta)\sqrt{\log n}} \\
\le & 2n^{-\frac{(a-\delta)^2}{2\sigma^2}+2}   \le {2}{n}^{-r} \to 0\qquad \text{ as } n \to \infty. 
\end{aligned}
\end{equation}
This and~\eqref{eq:goodCaseStep} imply the almost sure convergence assertion for $\thd = a\sqrt{\log n}$. 
\newline 
\noindent
c) \emph{Convergence in expectation}. It follows from~\eqref{eq:goodCaseStep} that
\begin{align*}
\E[a]{\norm{\hat{f}_n - f}_{L^p}^r} 
= & \E[a]{\norm{\hat{f}_n - f}_{L^p}^r; T_{\mathcal{I}}(y^n; f) \le a\sqrt{\log n}}\\
& {}\qquad\qquad + \E[a]{\norm{\hat{f}_n-f}_{L^p}^r; T_{\mathcal{I}}(y^n; f) > a\sqrt{\log n}} \\
\le & 2^{r/p}\Bigl(\frac{(2k_n+1)c}{n}\Bigr)^{\min\{r/2,r/p\}}\bigl(4(a+\delta)^2{\log n}\bigr)^{r/2}\bigl(1+o(1)\bigr) \\
& {}\qquad\qquad + \E[a]{\norm{\hat{f}_n-f}_{L^p}^r; T_{\mathcal{I}}(y^n; f) > a\sqrt{\log n}}. 
\end{align*}
We next show the second term above asymptotically vanishes faster. 
\begin{align}
&\E[a]{\norm{\hat{f}_n-f}_{L^p}^r; T_{\mathcal{I}}(y^n; f) > a\sqrt{\log n}}\nonumber \\
 = &\int_{0}^{2{n}^{p/2}}\Prob[a]{\norm{\hat{f}_n -f}_{L^p}^p \ge u; T_{\mathcal{I}}(y^n; f) > a\sqrt{\log n}}\frac{r}{p}u^{r/p -1}du\nonumber \\
& {}\qquad\qquad + \int_{2{n}^{p/2}}^{\infty}\Prob[a]{\norm{\hat{f}_n -f}_{L^p}^p \ge u; T_{\mathcal{I}}(y^n; f) > a\sqrt{\log n}} \frac{r}{p}u^{r/p -1} du\nonumber \\
\le & 2^{r/p}{n}^{r/2}\Prob[a]{T_{\mathcal{I}}(y^n; f) > a\sqrt{\log n}} + \int_{2{n}^{p/2}}^{\infty}\Prob[a]{\norm{\hat{f}_n -f}_{L^p}^p \ge u}  \frac{r}{p}u^{r/p -1} du\nonumber  \\
\le & 2^{r/p + 1} n^{-r/2} + \int_{2{n}^{p/2}}^{\infty}\Prob[a]{\norm{\hat{f}_n -f}_{L^p}^p \ge u}  \frac{r}{p}u^{r/p -1} du,\label{eq:tail_term}
\end{align} 
where the last inequality is due to~\eqref{eq:tailMRstat}. Introduce functions  $g = \sum_{i=0}^{n-1}y^n_i\Ind_{[i/n,(i+1)/n)}$ and $h = \sum_{i=0}^{n-1}f(i/n)\Ind_{[i/n,(i+1)/n)}$. Then, with notation $\xi^n \coloneqq \{\xi_i^n\}_{i = 0}^{n-1}$, $(x)_+ \coloneqq \max\{x, 0\}$ and $s\coloneqq (2r-p)_+$,  it holds that
\begin{align*}
\norm{\hat{f}_n -f}_{L^p}^p \le & 3^{(p-1)_+} \left( \norm{\hat{f}_n- g}_{L^p}^p + \norm{g - h}_{L^p}^p + \norm{h - f}_{L^p}^p \right)\\
\le & 3^{(p-1)_+} \left((a + \delta)^p{(\log n)}^{p/2} + n^{-1}\norm{\xi^n}_{\ell^p}^p + (2L)^p \right)\\
\le & 3^{(p-1)_+} \left((a + \delta)^p{(\log n)}^{p/2} + n^{-p/(p+s)}\norm{\xi^n}_{\ell^{p+s}}^{p} + (2L)^p \right).
\end{align*}
Thus, for large enough $n$, i.e.~if $n^{p/2} \ge 3^{(p-1)_+} \bigl((a + \delta)^p{(\log n)}^{p/2} + (2L)^p \bigr)$, 
\begin{align*}
& \int_{2{n}^{p/2}}^{\infty}\Prob[a]{\norm{\hat{f}_n -f}_{L^p}^p \ge u}  \frac{r}{p}u^{r/p -1} du \\
\le & \int_{2n^{p/2}}^{\infty}\Prob[a]{3^{(p-1)_+} \left((a + \delta)^p{(\log n)}^{p/2} + n^{-p/(p+s)}\norm{\xi^n}_{\ell^{p+s}}^{p} + (2L)^p \right) \ge u}  \frac{r}{p}u^{r/p -1} du  \\
\le & \int_{n^{p/2}}^{\infty}\Prob[a]{3^{(1+s/p)(p-1)_+}\frac{1}{n}\sum_{i=0}^{n-1}\abs{\xi_i^n}^{p+s} \ge u^{1+s/p}}\frac{r}{p}2^{r/p}u^{r/p -1} du \\
\le & 2^{r/p}3^{(1+s/p)(p-1)_+}\E[a]{\frac{1}{n}\sum_{i=0}^{n-1}\abs{\xi_i^n}^{p+s}}\int_{n^{p/2}}^{\infty}\frac{r}{p}u^{-(s-r)/p-2}du = \mathcal{O}(n^{-r/2}), 
\end{align*}
where the last inequality holds by the fact $s \ge 2r-p$ and $$
\E{\abs{\xi_i^n}^{p+s}} \le (p+s)2^{(p+s)/2}\sigma^{p+s}\Gamma\bigl(\frac{p+s}{2}\bigr) = \mathcal{O}(1) \qquad \text{for each } i.
$$
Thus, by~\eqref{eq:tail_term} it holds that
\begin{align*}
\E[a]{\norm{\hat{f}_n-f}_{L^p}^r; T_{\mathcal{I}}(y^n; f) > a\sqrt{\log n}} = & \mathcal{O}(n^{-r/2}) \\
= & o\Bigl(\left(n^{-1}(2k_n+1)\right)^{\min\{r/p,r/2\}}(\log n)^{r/2}\Bigr).
\end{align*}
This concludes the proof of Part (i). 

Moreover, we stress that for the choice of threshold $\thd = \thd(\beta)$ the assertions of Part~(i) still hold, which follow readily from the proof above, by noting the facts that $\thd(\beta) \le a\sqrt{\log n} $ for some constant $a$, due to~\eqref{eq:tailMRstat}, and that $\Prob[b]{T_{\mathcal{I}}(y^n; f) >\thd(\beta)}  = \mathcal{O}(n^{-r})$ by the choice of $\beta = \mathcal{O}(n^{-r})$. 

Finally, we consider Part (ii). The lower bound can be proven similarly as~\citet[Theorem 3.4]{LMS16}, by means of standard arguments based on testing many hypotheses \citep[pioneered by][]{IbHa77,Has78}. More precisely, we consider two collections of hypotheses 
\[
\set[g]{\sum_{i = 1}^{2k_n+2}\frac{(-1)^i\tilde z_0}{2}\Ind_{[\frac{i-1}{2k_n+2}+c_{i-1},\frac{i}{2k_n+2}+c_i)}}{c_i = \pm\frac{\sigma_0^2\log 2}{32n\tilde z_0^2}, c_0 = c_{2k_n+2} = 0}\subseteq \mathcal{S}_L(k_n)
\]
with $\tilde z_0 \coloneqq \min\{z_0, L\}$, and 
\[
\set[g]{\sum_{i = 1}^{k_n+1}\frac{(-1)^iL+c_i}{2}\Ind_{[\frac{i-1}{k_n+1},\frac{i}{k_n+1})}}{c_i = \pm\frac{\sigma_0}{4}\sqrt{\frac{k_n\log 2}{2 n}}}\subseteq \mathcal{S}_L(k_n). 
\]
Elementary calculation together with Fano's lemma~\citep[cf.][Corollary~2.6]{Tsy09} concludes the proof.

Parts (i) and (ii) imply that $\hat f_n$ is minimax optimal over $\mathcal{S}_{L}(k_n)$ up to a log-factor. Now the adaptation property follows by noting further that the choice of  the only tuning parameter $\thd$ in~\eqref{eq:defQ} is universal, i.e.~completely independent of the (unknown) true regression function. \hfill\qedsymbol

\subsection{Proof of Theorem~\ref{th:approx}}\label{app:approx}
The idea behind is that we first approximate the truth $f$ by a step function $f_{k_n}$ with $\mathcal{O}(k_n)$ jumps, and then treat $f_{k_n}$ as the underlying  ``true''  signal in model~\eqref{eq:model} (with additional approximation error). In this way, it allows us to employ similar techniques as in the proof of Theorem~\ref{th:step}. To be rigorous, we give a detailed proof as follows. 

\noindent
Since $\A_{q,L}^\gamma \subseteq \A_{\infty, L}^\gamma$, it is sufficient to consider $q < \infty$. 

\noindent
a) \emph{Good noise case}. Assume for the moment that the observations $y^n =\{y_i^n\}_{i=0}^{n-1}$ from model~\eqref{eq:model} are close to the truth $f$ in the sense that the event
\begin{equation}\label{eq:goodNoise}
 \mathcal{G}_n \coloneqq \left\{y^n\, :\, \sup_{I \in \mathcal{I}} \frac{1}{\sqrt{n\abs{I}}} \abs[B]{\sum_{i/n \in I}\bigl(y_i^n - \bar f_i^n\bigr)} -s_I\le a_0 \sqrt{\log n}\right\}
\end{equation}
holds  with $a_0 = \delta + \sigma\sqrt{2r+4}.$ Now let 
$$
k_n \coloneqq \Bigl\lceil\Bigl(\frac{4L}{a - a_0}\Bigr)^{2/(2\gamma+1)}\Bigl(\frac{n}{\log n}\Bigr)^{1/(2\gamma+1)} \Bigr\rceil.
$$
Since $f \in \A_{q,L}^{\gamma}$, for every $n$ there exists a step function $\tilde f_{k_n}\in \mathcal{S}$ with $\#J(\tilde f_{k_n}) \le k_n$ such that $\norm{f - \tilde f_{k_n}}_{L^q}  \le Lk_n^{-\gamma}$, by means of compactness argument. Based on the continuity of $\int_{[0,x)}f(t)dt$, one can find $\tau_0 \equiv 0 < \tau_1 < \cdots < \tau_{k_n} \equiv 1$ satisfying $\int_{[\tau_{i-1}, \tau_i)}\abs{f(t) - \tilde f_{k_n}(t)}^2 dt = \norm{f - \tilde f_{k_n}}^2_{L^2} / k_n$ for each $i$. By including such $\tau_i$'s as change-points, one can construct another step function $\breve{f}_{k_n}$ with $\#J(\breve{f}_{k_n}) \le 2k_n$, $\norm{f - \breve{f}_{k_n}}_{L^q} \le 2Lk_n^{-\gamma}$, and $\norm{(f - \breve{f}_{k_n})\Ind_{{I}}}_{L^2} \le 2Lk_n^{-\gamma-1/2}$ for every segment ${I}$ of $\breve{f}_{k_n}$. Moving each change-point of $\breve{f}_{k_n}$ to the closest point in $\{0, 1/n, \ldots, {(n-1)/n}\}$ but leaving the height of segments unchanged, one obtains a step function $f_{k_n}$ such that $\#J({f}_{k_n}) \le 2k_n$ and $\norm{f - {f}_{k_n}}_{L^q} \le 2Lk_n^{-\gamma} + 2L(k_n/n)^{1/q}$. Since $q \ge 2$, it holds that $\norm{(f - {f}_{k_n})\Ind_{{I}}}_{L^2} \le 2Lk_n^{-\gamma-1/2} + 2L n^{-1/2}$ for every segment ${I}$ of ${f}_{k_n}$. Then for sufficiently large $n$
\begin{align*}
T_{\mathcal{I}}(y^n; f_{k_n}) \le &  \sup_{\substack{I \in \mathcal{I} \\ f_{k_n} \equiv c_I \text{ on } I}} \frac{1}{\sqrt{n\abs{I}}} \abs[b]{\sum_{i/n \in I}( \bar f^n_i - c_I)} + \sup_{I \in \mathcal{I}} \frac{1}{\sqrt{n\abs{I}}} \abs[b]{\sum_{i/n \in I} (y_i^n - \bar f^n_i)} -s_I \\
\le &  \sup_{\substack{I \in \mathcal{I} \\ f_{k_n} \equiv c_I \text{ on } I}} \sqrt{\frac{n}{\abs{I}}} \int_I\abs{f(t) - f_{k_n}(t)}dt + a_0 \sqrt{\log n} \\
\le & \sup_{\substack{I \in \mathcal{I} \\ f_{k_n} \equiv c_I \text{ on } I}} \sqrt{n}\norm{(f - f_{k_n})\Ind_I}_{L^2} + a_0 \sqrt{\log n} \\
\le & 2n^{1/2}k_n^{-\gamma-1/2}L + 2L + a_0 \sqrt{\log n} \le a\sqrt{\log n}.  
\end{align*}
That is, $f_{k_n}$ lies in the constraint of~\eqref{eq:mr_segment}. Thus, by definition, $\# J(\hat{f}_n) \le \#J(f_{k_n}) \le 2k_n$. Let intervals $\{I_i\}_{i = 0}^m$ be the partition of $[0,1)$ by $J(\hat{f}_n)\cup J(f_{k_n})$ with $m \le 4 k_n$. Then
\[
\norm{\hat{f}_n - f_{k_n}}_{L^p}^p = \sum_{i = 0}^m \abs{\hat{\theta}_i - \theta_i}^p \abs{I_i} \qquad \text{ with } \hat{f}_n|_{I_i} \equiv \hat\theta_i \text{ and } {f_{k_n}}|_{I_i} \equiv \theta_i. 
\]
If $\abs{I_i} > c/n$, there is $\tilde{I}_i \in \mathcal{I}$ such that $\tilde{I}_i \subseteq I_i$ and $\abs{\tilde{I}_i} \ge \abs{I_i}/c$. Then,
\[
{\abs[b]{\tilde{I}_i}}^{1/2}\abs[g]{\theta - \frac{1}{n\abs{\tilde{I}_i}}\sum_{j/n \in \tilde{I}_i}y_j^n} \le (a+\delta)\sqrt{\frac{\log n}{n}}\qquad\text{ for } \theta = \theta_i \text{ or }\hat\theta_i,
\]
which, together with $\abs{\tilde{I}_i} \ge \abs{I_i}/c$, implies
\[
\abs{I_i}^{1/2}\abs{\hat\theta_i -\theta_i} \le 2(a+\delta)\sqrt{\frac{c\log n}{n}}.
\]
If $\abs{I_i} \le c/n$, then we have for some $i_0$
\begin{align*}
 \abs{\hat\theta_i} \le \abs{\hat\theta_i - y_{i_0}^n} + \abs[b]{y_{i_0}^n -\bar f_{i_0}^{n}} + \norm{f}_{L^\infty} \le 2(a+\delta)\sqrt{\log n} + L
\text{ and }
\abs{\theta_i} \le \norm{f}_{L^{\infty}} \le L, 
\end{align*}
which lead to
\[
\abs{\hat\theta_i -\theta_i} \le \abs{\hat\theta_i}+\abs{\theta_i} \le 2(a+\delta)\sqrt{\frac{\log n}{n}} + 2L.
\]
Thus, by combining these two situations, we obtain that
\begin{multline*}
\norm{\hat{f}_n - f_{k_n}}_{L^p}^p \le \sum_{i:\abs{I_i} > c/n} \biggl(2(a+\delta)\sqrt{\frac{c\log n}{n\abs{I_i}}}\biggr)^p\abs{I_i} 
+ \sum_{i:\abs{I_i} \le c/n}\left(2(a+\delta)\sqrt{\log n}+2 L\right)^p\frac{c}{n}. 
\end{multline*}
Then, with a similar argument as for~\eqref{eq:goodCaseStep}, we obtain as $n\to\infty$ 
\[
\norm{\hat{f}_n-f_{k_n}}_{L^p}^p \le 2 \Bigl(4(a+\delta)^2{\log n}\Bigr)^{p/2}\Bigl(\frac{(4k_n+1)c}{n}\Bigr)^{\min\{1,p/2\}}\bigl(1+o(1)\bigr),
\]
which together with a triangular inequality leads to 
\begin{equation}\label{eq:goodCaseApprox}
\norm{\hat{f}_n-f}_{L^p}^r \le 2^{(2/p+1)r} \Bigl(4(a+\delta)^2{\log n}\Bigr)^{r/2}\Bigl(\frac{(4k_n+1)c}{n}\Bigr)^{\min\{r/p,r/2\}}\bigl(1+o(1)\bigr).
\end{equation}
\newline 
\noindent
b) \emph{Rates of convergence}. The rate of almost convergence is a consequence of~\eqref{eq:goodCaseApprox} and the fact that, due to~\eqref{eq:tailMRstat},
\begin{align*}
\limsup_{n\to\infty}\Prob[a]{\mathcal{G}_n^c} 
\le  \limsup_{n\to\infty}\Prob[a]{\sup_{I \in \mathcal{I}}\frac{1}{\sqrt{n\abs{I}}}\abs[B]{\sum_{i/n \in I} \xi_i^n} > (a_0-\delta)\sqrt{\log n}} = 0.
\end{align*}
Similar to the proof step (iii) of Theorem~\ref{th:step}, we drive from~\eqref{eq:goodCaseApprox} that, as $n \to \infty$, 
\begin{align*}
& \E[a]{\norm{\hat{f}_n - f}_{L^p}^r}\\
 = &\E[a]{\norm{\hat{f}_n - f}_{L^p}^r; \mathcal{G}_n} + \E[a]{\norm{\hat{f}_n - f}_{L^p}^r; \mathcal{G}_n^c}\\
\le &\E[a]{\norm{\hat{f}_n - f}_{L^p}^r; \mathcal{G}_n} + 2^{r/p}{n}^{r/2}\Prob[a]{\mathcal{G}_n^c} + \int_{2n^{p/2}}^{\infty} \Prob[a]{\norm{\hat{f}_n-f}_{L^p}^p\ge u}\frac{r}{p}u^{r/p -1} du \\
\le & \mathcal{O}\Bigl(({\log n})^{r/2}\bigl(n^{-1}{k_n}\bigr)^{\min\{r/p,r/2\}}\Bigr) + \mathcal{O}\bigl(n^{-r/2}\bigr)\\
 = & \mathcal{O}\Bigl(({\log n})^{r/2}\bigl(n^{-1}{k_n}\bigr)^{\min\{r/p,r/2\}}\Bigr),
\end{align*}
which shows the rate of convergence in expectation. 

We note, in addition, that for the choice of threshold $\thd = \thd(\beta)$, the proof follows in the same way as above, based on the facts that 
$\thd(\beta) \le a\sqrt{\log n} $ for some constant $a$, due to~\eqref{eq:tailMRstat}, and that $\Prob[a]{\mathcal{G}_n^c}  = \mathcal{O}(n^{-r})$ by the choice of $\beta = \mathcal{O}(n^{-r})$. 
\hfill\qedsymbol

\subsection{{Proof of Theorem~\ref{th:feature}}}\label{app:feature}
{
The proof relies on the following lemma. 
\begin{lemma}\label{lm:feature}
Under model~\eqref{eq:model} with the truth $f \in \D$, let $\hat{f}_n$ be a multiscale change-point segmentation method in Definition~\ref{def:mrseg} with constants $c$, $\delta$,  interval system $\mathcal{I}$, and universal threshold $\eta$ in~\eqref{eq:defQ} or~\eqref{eq:refineQ}. 
\begin{itemize}
\item[\emph{a)}] 
Let $\mathcal{I}_n$ be an arbitrary collection of (possibly random) intervals. If further $f \in \A^{\gamma}_{2,L}$ for some $\gamma, L > 0$, then  
$$
\lim_{n\to\infty}\Prob[a]{\max\set[b]{\abs{I}^{1/2}\abs{m_I(\hat f_n) - m_I(f)}}{I \in \mathcal{I}_n}  \le C\Bigl(\frac{\log n}{n}\Bigr)^{\gamma/(2\gamma + 1)}} = 1,
$$
where $C$ is a constant depending only on $\eta$, $c$ and $L$. 
\item[\emph{b)}]
If further $\mathcal{I}_n \subseteq \mathcal{I}$, and on each $I \in \mathcal{I}_n$ we have $\hat f_n$ is constant, then
$$
\Prob[B]{\abs{I}^{1/2}\abs{m_I(\hat f_n) - m_I(f)}\le\frac{2(\eta +s_I)}{n^{1/2}}\quad \text{ for all }I \in \mathcal{I}_n} \ge \Prob[a]{T_{\mathcal{I}}(\xi^n; 0) \le \eta},
$$
where the right hand side converges to $1$ as $n\to\infty$. 
\end{itemize}
\end{lemma}
}

\begin{proof}{
Part a): Note that for each $I \in \mathcal{I}_n$, 
\begin{multline*}
\abs{I}^{1/2}\abs{m_I(\hat f_n) - m_I(f)} \le \frac{1}{{\abs{I}}^{1/2}}\int_{I} \abs{\hat f_n(x) - f(x)} dx \\
 \le \frac{1}{{\abs{I}}^{1/2}} \abs{I}^{1/2}\Bigl(\int_{I}\abs{\hat f_n(x) - f(x)}^2\Bigr)^{1/2} \le \norm{\hat f_n - f}_{L^2}. 
\end{multline*}
Then, the assertion follows from Theorem~\ref{th:approx}. }

{
Part b): Assume $T_{\mathcal{I}}(\xi^n; 0) \le \eta$. Then $T_{\mathcal{I}}(y^n; f) \le \eta$. Since $T_{\mathcal{I}}(y^n; \hat f_n) \le \eta$ by definition, we obtain for either $g = f$ or $g = \hat f_n$ 
$$
\abs{I}^{1/2}\abs[B]{m_I(g) -\frac{1}{n\abs{I}}\sum_{j/n \in I} y_j^n} \le s_I + \eta\qquad \text{ for any } I \in \mathcal{I}_n.
$$
Thus, by triangular inequality it holds $\abs{I}^{1/2}\abs[b]{m_I(\hat f_n) -m_I(f)} \le 2(s_I + \eta)$. This shows  
$$
\left\{T_{\mathcal{I}}(\xi^n; 0) \le \eta\right\} \subseteq \Bigl\{\abs{I}^{1/2}\abs{m_I(\hat f_n) - m_I(f)}\le\frac{2(\eta +s_I)}{n^{1/2}}\quad \text{ for all }I \in \mathcal{I}_n\bigr\}, 
$$
which shows the assertion. By the choice of $\eta$, it holds $\lim_{n\to\infty}\Prob[a]{T_{\mathcal{I}}(\xi^n; 0)} = 1$. }
\end{proof}

{
\emph{Part} (i): We select $\mathcal{I}_n$ as a fixed collection of intervals that capture the modes and troughs of $f$. That is, $\mathcal{I}_n \coloneqq \left\{I_1, \ldots, I_m\right\}$ for some $m$ such that $I_1 < I_2 < \cdots < I_m$ and $m_{I_1}(f) \neq m_{I_2}(f) \neq \cdots \neq m_{I_m}(f)$. By Lemma~\ref{lm:feature} a) we have
$$
\max\set[b]{\abs{I}^{1/2}\abs{m_I(\hat f_n) - m_I(f)}}{I \in \mathcal{I}_n} \to 0.
$$
It implies $m_{I_1}(\hat f_n) \neq m_{I_2}(\hat f_n) \neq \cdots \neq m_{I_m}(\hat f_n)$ for sufficiently large $n$, and thus~\eqref{eq:mt}.
}

{
Now we set $\mathcal{I}_n\coloneqq \set[b]{[x, x+\lambda_n^{\varepsilon}), [x -\lambda_n^{\varepsilon}, x)}{x \in J_{\varepsilon}(f)}$ with $$
\lambda_n^{\varepsilon} \coloneqq \min\bigl\{d\bigl(J(\hat f_n), J_\varepsilon(f)\bigr), \, \delta_n\bigr\}\qquad \text{ for some positive }  \delta_n\to 0\text{ arbitrarily slow.}
$$ 
For $x \in J_{\varepsilon}(f)$, note that $\hat f_n$ is constant on $[x  -\lambda_n^{\varepsilon}, x + \lambda_n^{\varepsilon})$, which in particular implies $m_{[x -\lambda_n^{\varepsilon}, x)}(\hat f_n) = m_{[x, x+\lambda_n^{\varepsilon})}(\hat f_n)$.  Moreover, as $\lambda_n^{\varepsilon} \to 0$, from the definition of $\Delta^\varepsilon_f$ and  $f \in \mathcal{D}$ it follows for sufficiently large $n$ 
\begin{equation}\label{eq:szjmp}
\abs[a]{m_{[x -\lambda_n^{\varepsilon}, x)}(f) - m_{[x, x+\lambda_n^{\varepsilon})}(f)} \ge \frac{1}{2}\Delta^\varepsilon_f\qquad \text{ for all }x \in J_{\varepsilon}(f).  
\end{equation}
We claim that for each $x \in J_{\varepsilon}(f)$ there exists $I_x = [x, x+\lambda^\varepsilon_n)$ or $[x-\lambda^\varepsilon_n, x)$ such that $\abs{m_{I_x} (f) - m_{I_x}(\hat f_n)} \ge \Delta^\varepsilon_f/4$. Otherwise, if $\abs{m_{I_x} (f) - m_{I_x}(\hat f_n)} < \Delta^\varepsilon_f/4$ holds for both $I_x = [x, x+\lambda^\varepsilon_n)$ and $[x-\lambda^\varepsilon_n, x)$, then it leads to $\abs[a]{m_{[x -\lambda_n^{\varepsilon}, x)}(f) - m_{[x, x+\lambda_n^{\varepsilon})}(f)} < \Delta^\varepsilon_f/2$, which contradicts with~\eqref{eq:szjmp}. Thus, by Lemma~\ref{lm:feature} a), it holds
$$
\lim_{n\to\infty}\Prob[a]{\frac{\Delta^\varepsilon_f}{4} \le \abs{m_{I_x} (f) - m_{I_x}(\hat f_n)} \le \frac{C}{\sqrt{\lambda_n^{\varepsilon}}}\Bigl(\frac{\log n}{n}\Bigr)^{\gamma/(2\gamma + 1)}\quad\text{ for all } x\in J_{\varepsilon}(f)} = 1.
$$
It implies $\lambda_n^{\varepsilon} \le 16C^2(\Delta^\varepsilon_f)^{-2}(\log n / n)^{2\gamma/(2\gamma+1)} $ almost surely, as $n\to \infty$. By letting $\delta_n\to 0$ slower than $(\log n / n)^{2\gamma/(2\gamma+1)}$, we obtain
$$
\lim_{n\to\infty}\Prob[g]{d\bigl(J(\hat f_n), J_\varepsilon(f)\bigr) \le \frac{16C^2}{(\Delta^\varepsilon_f)^2}\Bigl(\frac{\log n}{n}\Bigr)^{2\gamma/(2\gamma + 1)}}=1,
$$
and thus $\lim_{n\to\infty} \Prob[b]{\#J(\hat{f}_n) \ge \#J_\varepsilon(f)} = 1$. Therefore, \eqref{eq:jmp} holds.
}

{
\emph{Part} (ii): By Lemma~\ref{lm:feature} b), we have 
$$
\Prob[B]{\abs{I_i}^{1/2}\abs{m_{I_i}(\hat f_n) - m_{I_i}(f)}\le\frac{2(\eta +s_{I_i})}{n^{1/2}}\, \text{ for }i = 1, 2} \ge \Prob[a]{T_{\mathcal{I}}(\xi^n; 0) \le \eta(\beta)} \ge 1 -\beta.
$$
Note that $\abs{I_i}^{1/2}\abs{m_{I_i}(\hat f_n) - m_{I_i}(f)}\le\frac{2(\eta +s_{I_i})}{n^{1/2}}$ for $i = 1,2$ and~\eqref{eq:sft} imply
\begin{align*}
m_{I_1}(f) - m_{I_2}(f) \ge & \,m_{I_1}(\hat f_n)-m_{I_2}(\hat f_n)-\sum_{i = 1}^2\abs{m_{I_i}(\hat f_n) - m_{I_i}(f)} \\
>  &\,\frac{2\bigl(\thd(\beta)+s_{I_1}\bigr)}{\sqrt{n\abs{I_1}}} + \frac{2\bigl(\thd(\beta)+s_{I_2}\bigr)}{\sqrt{n\abs{I_2}}}-\sum_{i = 1}^2\abs{m_{I_i}(\hat f_n) - m_{I_i}(f)} \ge 0.
\end{align*}
This concludes the proof. \hfill\qedsymbol
}

\subsection{Proof of Theorem~\ref{th:oSeg}}\label{app:oSeg}
{For simplicity, we assume that the noise $\xi^n_i$ has homogeneous variance $\sigma^2_0$, since for the general case it is obvious to modify the following proof accordingly.} For every $\tau\equiv(\tau_0, \tau_1,\ldots,\tau_k) \in \Pi_n$, we define $\#\tau \coloneqq k$, and by elementary calculation obtain 
$$
\E{\norm{\hat f_{\tau, n} -f}^2_{L^2}} = \norm{s_{\tau}-f}_{L^2}^2 + \frac{\#\tau}{n}\sigma_0^2
$$
where $s_\tau$ is the best $L^2$-approximant of $f$ with change-points specified by $\tau$. Define $\tau_*\equiv\tau_*(n)\in\Pi_n$ such that 
$\E{\norm{\hat f_{\tau_*,n} -f}_{L^2}^2} = \inf_{\tau\in\Pi_n}\E{\norm{\hat f_{\tau,n} -f}_{L^2}^2}$. Now we claim that there exists a constant $C$ satisfying 
\begin{equation}\label{eq:claim}
\norm{s_{\tau_*} -f}^2_{L^2} \le C \frac{\#\tau_*}{n}\sigma_0^2 \qquad \text{for sufficiently large }n.
\end{equation} 

To prove the claim~\eqref{eq:claim}, we, anticipating contradiction, assume that 
$$
\limsup_{n\to\infty}\frac{n\norm{s_{\tau_*}-f}_{L^2}^2}{\#\tau_*\sigma_0^2} = \infty.
$$ 
One can choose $m\equiv m(n)$ such that $\limsup_{n\to\infty} n\norm{s_{\tau_*} -f}^2_{L^2}(m\#\tau_*\sigma_0^2)^{-1} = \infty$, and  $\lim_{n\to\infty}m = \infty$.  Define $\upsilon_*$ as $\norm{s_{\upsilon_*} - f}_{L^2} = \inf_{\upsilon \in U_{\tau_*,m}}\norm{s_{\upsilon} -f}_{L^2}$ with 
$$
U_{\tau_*,m}\coloneqq\set{\upsilon\in\Pi_n}{\upsilon\equiv(0,\upsilon^1_1,\ldots,\upsilon^1_m \equiv \tau_*^1, \ldots,\upsilon_1^k,\ldots,\upsilon_m^k\equiv\tau_*^k)\text{ if }\tau_* \equiv (0,\tau_*^1,\ldots,\tau_*^k)}.
$$
It follows from $m\to\infty$ and $f\in\A^\gamma_{2}\cap L^\infty$ for some $\gamma$ that $\norm{s_{\upsilon_*}-f}_{L^2}/\norm{s_{\tau_*}-f}_{L^2}\to 0$. Then we obtain
$$
\limsup_{n\to \infty}\frac{\E{\norm{\hat f_{\tau_*,n}-f}_{L^2}^2}}{\E{\norm{\hat f_{\upsilon_*,n}-f}_{L^2}^2}}\ge\limsup_{n\to\infty}\frac{\norm{s_{\tau_*} - f}_{L^2}^2}{\norm{s_{\upsilon_*} - f}_{L^2}^2 + m\#\tau_*\sigma_0^2/n} = \infty,
$$
which contradicts the definition of $\tau_*$.

Denote $L\coloneqq \norm{f}_{L^\infty}$. Similar to part a) in the proof of Theorem~\ref{th:approx}, one can construct a step function $\tilde s_{\tau_*}$, by adding another $\#\tau_* $ change-points to $s_{\tau_*}$ and later shifting all the change-points to the grid points $i/n$, such that $\# J(\tilde s_{\tau_*}) \le 2(\#\tau_*-1)$, $\norm{\tilde s_{\tau_*} - f}_{L^2}^2 \le 2\norm{s_{\tau_*} -f}_{L^2}^2 + 2n^{-1}\#\tau_*L^2$, and $\norm{(\tilde s_{\tau_*} -f)\Ind_I}^2_{L^2} \le 2(\#\tau_*)^{-1}\norm{s_{\tau_*} -f}_{L^2}^2 + 2n^{-1}L^2$ for each segment $I$ of $\tilde s_{\tau_*}$. 

Assume now the ``good noise'' case, namely, event $\mathcal{G}_n$ in \eqref{eq:goodNoise} holds true, and that $\thd$ is defined in~\eqref{eq:defQ}. Then we have for sufficiently large $n$,
\begin{align*}
T_{\mathcal{I}}(y^n; \tilde s_{\tau_*}) \le &  \sup_{\substack{I \in \mathcal{I} \\ \tilde s_{\tau_*} \equiv c_I \text{ on } I}} \frac{1}{\sqrt{n\abs{I}}} \abs[b]{\sum_{i/n \in I}( \bar f^n_i - c_I)} + \sup_{I \in \mathcal{I}} \frac{1}{\sqrt{n\abs{I}}} \abs[b]{\sum_{i/n \in I} (y_i^n - \bar f^n_i)} -s_I \\
\le & \sup_{\substack{I \in \mathcal{I} \\ \tilde s_{\tau_*} \equiv c_I \text{ on } I}} \sqrt{n}\norm{(f - \tilde s_{\tau_*})\Ind_I}_{L^2} + a_0 \sqrt{\log n} \\
\le & \sqrt{2n(\#\tau_*)^{-1}\norm{s_{\tau_*} -f}_{L^2}^2 + 2L^2} + a_0 \sqrt{\log n} \\
\le & \sqrt{2C\sigma_0^2 + 2L^2} + a_0 \sqrt{\log n}  \le a\sqrt{\log n},
\end{align*}
where $C$ is the constant in \eqref{eq:claim}. Again following similar lines of part a) in the proof of Theorem~\ref{th:approx}, one can obtain that 
$$
\norm{\hat f_n - \tilde s_{\tau_*}}_{L^2}^2 \le 32(a+\delta)^2 c\log n \frac{\#\tau_*}{n}(1+o(1)). 
$$
It further follows that 
\begin{align*}
\norm{\hat f_n -f}_{L2}^2 & \le  2\norm{f - \tilde s_{\tau_*}}_{L^2}^2 + 2\norm{\hat f_n - \tilde s_{\tau_*}}^2_{L^2}  \\
& \le 4\norm{s_{\tau_*}-f}_{L^2}^2 + 4L^2\frac{\tau_*}{n} + 64(a+\delta)^2 c\log n \frac{\#\tau_*}{n}(1+o(1)). 
\end{align*}
Thus, under event $\mathcal{G}_n$, we obtain for large enough $n$
$$
\norm{\hat f_n -f}^2_{L^2} \le \tilde C \log n \bigl( \norm{s_{\tau_*}-f}_{L^2}^2 + \frac{\#\tau_*}{n}\sigma_0^2\bigr) \le \tilde C \log n\E{\norm{\hat f_{\tau_*, n} -f}^2_{L^2}},
$$
where $\tilde C$ is a constant independent of $f$. 

The assertion of the theorem is then followed by applying the same technique as in the part b) of the proof of Theorem~\ref{th:approx}. Again, as in Theorem~\ref{th:approx}, the above argument remains valid if we set $\thd=\thd(\beta)$ as in~\eqref{eq:refineQ}.
\hfill\qedsymbol


\end{document}